\documentclass[letterpaper]{article}

\usepackage[letterpaper]{geometry}
\usepackage{amsmath, amsthm, amsfonts, amsbsy, thmtools, amssymb, bm,textcmds,mathscinet}
\usepackage[sans]{dsfont}
\usepackage{hyperref}
\usepackage[mathscr]{euscript}
\usepackage{tikz}
\usetikzlibrary{arrows}
\usepackage{cleveref}
\usepackage{tikz}
\usepackage{comment}
\usepackage{enumitem}

\numberwithin{equation}{section}

\DeclareMathOperator{\Var}{Var}

\DeclareMathOperator{\Adm}{Adm}
\DeclareMathOperator{\dist}{dist}

\DeclareMathOperator{\diam}{diam}
\DeclareMathOperator{\TV}{TV}

\newcommand{\ii}{\mathbf i}
\newcommand{\eps}{\varepsilon}
\newcommand{\E}{\mathbb E}

\newtheorem{theorem}{Theorem}
\newtheorem{prop}[theorem]{Proposition}
\newtheorem{lem}[theorem]{Lemma}
\newtheorem{lemma}[theorem]{Lemma}

\theoremstyle{remark}
\newtheorem{remark}[theorem]{Remark}
\theoremstyle{definition}
\newtheorem{definition}[theorem]{Definition}

\newtheorem{example}[theorem]{Example}

\title{Gaussian Unitary Ensemble in random lozenge tilings}

\author{Amol Aggarwal \and Vadim Gorin}

\date{}

\begin{document}
	
\maketitle

\begin{abstract}
 This paper establishes a universality result for scaling limits of uniformly random lozenge tilings of large domains. We prove that whenever a boundary of the domain has three adjacent straight segments inclined under 120 degrees to each other, the asymptotics of tilings near the middle segment is described by the GUE--corners process of random matrix theory. An important step in our argument is to show that fluctuations of the height function of random tilings on essentially arbitrary simply-connected domains of diameter $N$ have smaller magnitude than $N^{1/2}$.
\end{abstract}

\tableofcontents

\section{Introduction}

Random lozenge tilings attracted substantial interest in the recent years following a 25 years old mathematical discovery that they exhibit highly non-trivial limit shapes with flat ``frozen'' facets; see  \cite{Vershik_Vienna,CLP,CerfKenyon,okounkov_reshetikhin_3d} for the pioneer results, \cite{Gorin_book} for an extensive review of the area, and Figures \ref{Figure_hexagon} and \ref{Figure_GUE_cases} for simulations.

One intriguing feature of random tilings is their link to random matrix theory. The definition of the random tilings model does not have any obvious matrices involved,
yet, the answers to the asymptotic questions about tilings turn out to involve distributions which are also encountered in the studies of eigenvalues of random matrices. In this paper we concentrate on the most direct of such connections and discuss how local limits near straight boundaries of tilings of general domains are described by the GUE--corners process --- a multilevel version of the celebrated Gaussian Unitary Ensemble.

\begin{definition} \label{Def_GUE_corners}
  Let $X=[X_{ij}]_{i,j=1}^{\infty}$ be an infinite matrix of i.i.d.\ complex Gaussian random variables of the form $N(0,1)+\ii N(0,1)$ and set $M=\frac{X+X^*}{2}$ to be its Hermitian part. Define ${\xi_1^k\le \xi_2^k\le \dots\le \xi_k^k}$ to be the eigenvalues of the principal $k\times k$ top-left corner of $M$ as in \eqref{eq_corners}. The array $\{\xi_i^k\}_{1\le i \le k}$ is called the \emph{GUE--corners process}.
\begin{equation}\label{eq_corners}\left(
\begin{array}{cccc}
M_{11} & \multicolumn{1}{|c}{M_{12}}& \multicolumn{1}{|c}{M_{13}} &\multicolumn{1}{|c}{M_{14}} \\
\cline{1-1} M_{21} & M_{22} & \multicolumn{1}{|c}{M_{23}}& \multicolumn{1}{|c}{M_{24}}\\
\cline{1-2} M_{31}& M_{32} & M_{33} &\multicolumn{1}{|c}{M_{34}}\\
\cline{1-3} M_{41}& M_{42} & M_{43} & M_{44}
\end{array}\right)
\end{equation}
\end{definition}

\begin{figure}[t]
\center
  \includegraphics[width=0.4\textwidth]{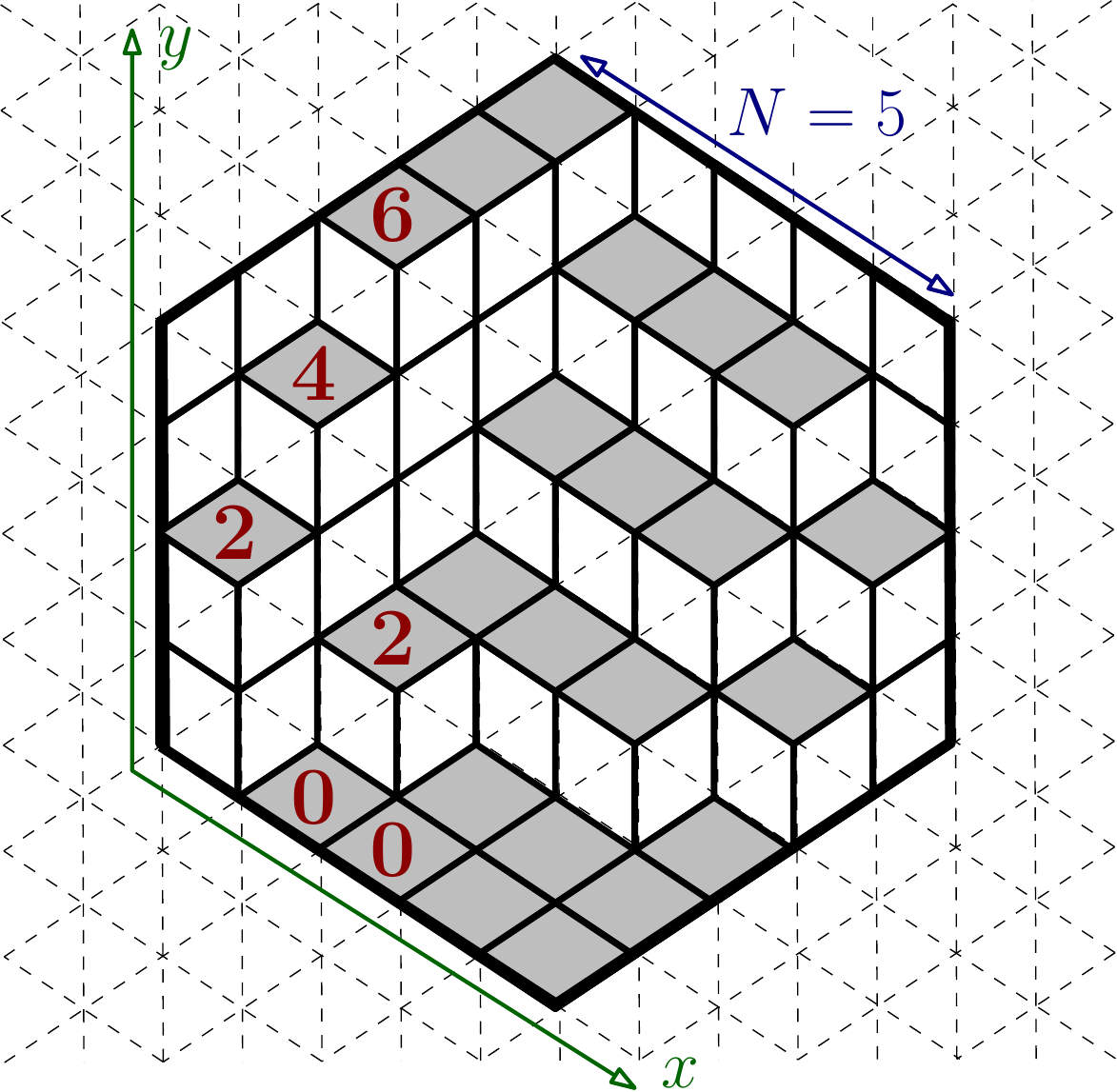} \qquad \qquad \qquad
  \includegraphics[width=0.4\textwidth]{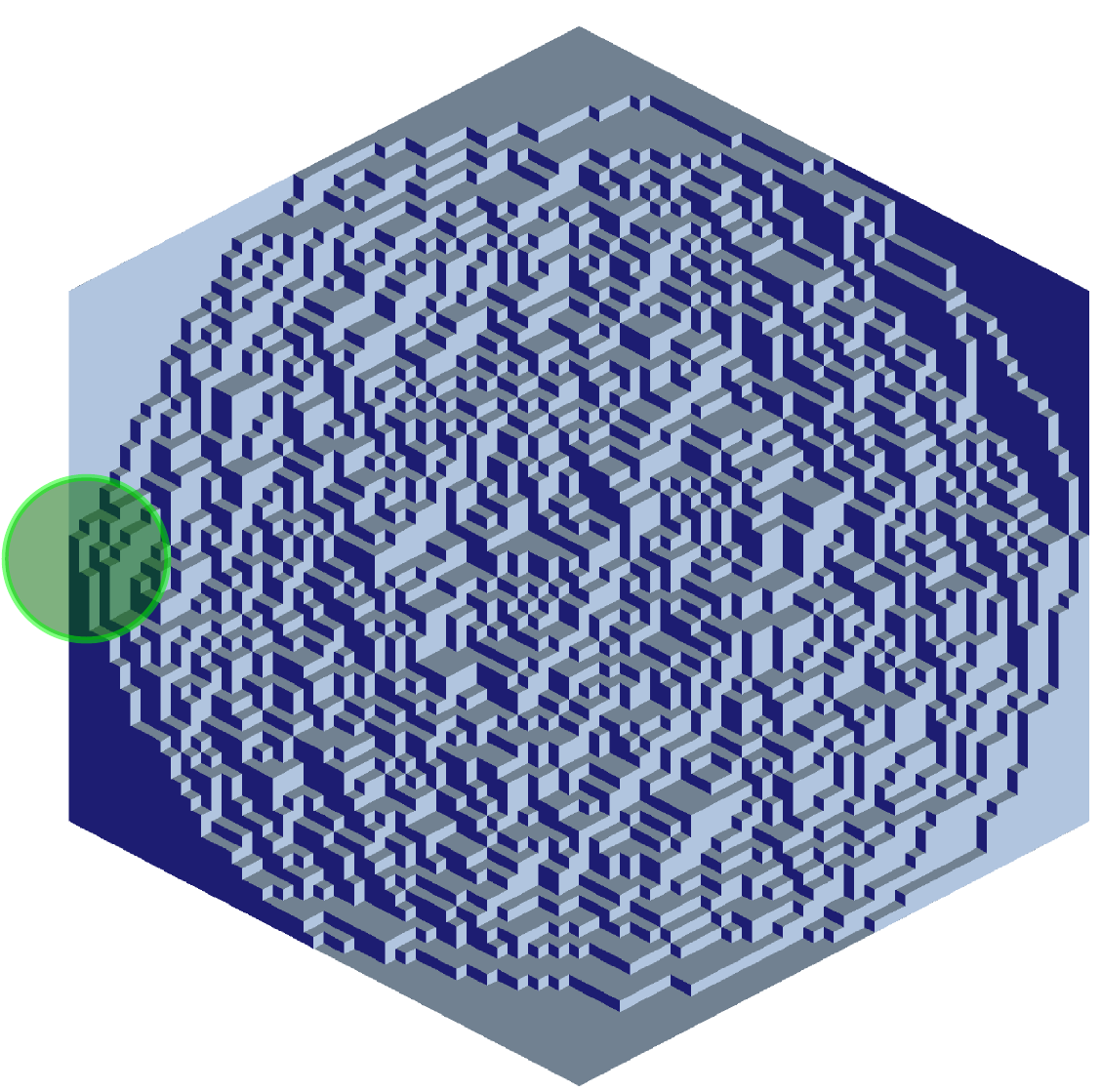}
\caption{Left panel: A lozenge tiling of a regular hexagon and coordinates of horizontal lozenges along the left boundary: $y_1^1=2$ on the first vertical, $(y_1^2, y_2^2)=(0,4)$ on the second vertical, $(y_1^3, y_2^3, y_3^3)=(0,2,6)$ on the third vertical. Right panel: A uniformly random tiling of the hexagon for $N=50$ and a point where we see the GUE--corners process. Three types of lozenges are shown in three different colors. }
\label{Figure_hexagon}
\end{figure}

Let us explain the connections between tilings and the GUE--corners process on the example of uniformly random tilings of a hexagon, where it was first proved in \cite{JN,Nordenstam}. Consider a regular hexagon of side length $N$ drawn on the triangular grid. We tile the hexagon with three types of lozenges (rhombuses) obtained by gluing two adjacent triangles on the grid; one possible tiling for $N=5$ is shown in the left panel of Figure \ref{Figure_hexagon}. For fixed $N$ there are finitely many such tilings and we are interested in a uniformly random lozenge tiling of the hexagon. A fascinating feature of random tilings is how their boundary conditions lead to inhomogeneous density profiles for three types of lozenges. The right panel of Figure \ref{Figure_hexagon} shows a simulation for a large hexagon, where this aspect is clearly visible. In particular, one sees that inside the inscribed circle all three types of lozenges are present, while outside it there are six regions with only one type of lozenges in each. The points where the circle is tangent to the sides of the hexagon serve as the most drastic exemplification of this phenomenon since, even locally, the density of tiles there is maximally inhomogeneous. Simultaneously, these are the only points through which the tilings non-trivially interact with boundaries.  Therefore, a question of interest is to mathematically understand these special points and to figure out the conditions of their appearance.

Let us distinguish one type of lozenge, say, we focus on the horizontal lozenges shown in gray in Figure \ref{Figure_hexagon} and trace their positions near the left vertical boundary of the hexagon. Due to combinatorics of the model, there will be one horizontal lozenge with coordinate $y_1^1$ on the first vertical line  (in our coordinate system shown in left panel of Figure \ref{Figure_hexagon}, a lozenge at the lower boundary of the hexagon has coordinate $0$). Further, there will be two horizontal lozenges with coordinates $y_1^2< y_2^2$ on the second vertical line, three horizontal lozenges with coordinates $y_1^3<y_2^3<y_3^3$ on the third line, etc. Since the tiling is random, $y_i^k$ are random variables. We are interested in their asymptotic behavior as $N\to \infty$. \cite{JN,Nordenstam} showed that the limit is governed by the object of Definition \ref{Def_GUE_corners}:
\begin{equation}
\label{eq_hex_to_GUE}
 \lim_{N\to\infty} \left( \frac{y_i^k- m_N}{\sqrt{N}}\right)_{1\le i \le k} = \bigl(\sigma \xi_i^k\bigr)_{1\le i \le k},\qquad  \text{ where }m_N=\tfrac{N}{2}\quad \text{ and }\quad \sigma=\sqrt{\tfrac{3}{8}},
\end{equation}
in the sense of convergence of finite-dimensional distributions (jointly over all $i$ and $k$ such that $1\le i \le k$). In particular, \eqref{eq_hex_to_GUE} implies that $y_1^1$ is asymptotically Gaussian as $N\to\infty$ and its variance grows proportionally to $\sqrt{N}$. In words, \eqref{eq_hex_to_GUE} says that discontinuity in densitites of lozenges near a boundary leads to the appearance of a random matrix object.

\bigskip

Beyond tilings of the hexagon, convergence to the GUE--corners process near the boundaries, generalizing \eqref{eq_hex_to_GUE}, was proven for several classes of domains; see \cite{OR_Birth,GP,Novak,Panova,Mk_Petrov}. The only necessary change is the expression for the constants $m_N$ and $\sigma$. Simultaneously, Okounkov and Reshetikhin \cite{OR_Birth} explained  that if the scaling limit of tilings near a boundary exists and satisfies certain very natural assumptions, then by invoking a classification theorem for Gibbs measures on interlacing arrays, one can conclude that the only possible candidate is the GUE--corners process. Hence, a universality prediction arose in \cite{JN,OR_Birth}: whenever a domain's boundary has a straight segment, the GUE--corners process should asymptotically arise near this segment.

Our main result, Theorem \ref{Theorem_GUE_simple}, proves this prediction. While a possible path to a proof might have been through checking that the assumptions of \cite{OR_Birth} always hold, we did not know how to make such a check and found another approach. In our proof, we rely on the results of \cite{GP, Novak} which show that the GUE--corners process always appears in tilings of a special class of domains called trapezoids. We further observe that locally, near a straight segment of the boundary, any planar domain looks like a trapezoid. However (in contrast to the setting of \cite{GP, Novak}) one boundary of the trapezoid is random and its fluctuations could a priori spoil the GUE--corners asymptotic. To resolve this issue, it must be shown that these fluctuations are $o(N^{1/2})$, that is, they have smaller scale than the GUE--corners limit --- while widely expected\footnote{See Lectures 11-12 in \cite{Gorin_book} for the heuristics explaining that the fluctuations should grow logarithmically in the size of the domain.}, such a bound for tilings of general domains has not been rigorously proven up to now.
Known concentration estimates on arbitrary domains only apply on scale $O(N^{1/2})$; proving the improved $o(N^{1/2})$ estimate occupies a substantial portion of the manuscript. This is done by first using a (known) interpretation for the height variance through the double-dimer model, and then by patching together multiple instances of bulk local limit results of \cite{ULS} for random lozenge tilings of arbitrary domains.

\bigskip

We expect that universal appearance of the GUE--corners process extends beyond lozenge tilings with the first candidates for probing the extensions given by the domino tilings and the six-vertex model, see \cite{JN,Gorin_6v,Dimitrov,DR} for partial results in this direction. An interesting and important project for a future work is to figure out whether our approach can be helpful beyond lozenges.

\subsection*{Acknowledgements}
The authors are grateful to David Keating and Ananth Sridhar for the simulation for Figure \ref{Figure_GUE_cases}. We also thank Alexander I. Bufetov for bringing the work \cite{CTFMD} to our attention.
The work of A.A.\ was partially supported by a Clay Research Fellowship. The work of V.G.\ was partially supported by NSF Grants DMS-1664619, DMS-1949820, by BSF grant 2018248, and by the Office of the Vice Chancellor for Research and Graduate Education at the University of Wisconsin--Madison with funding from the Wisconsin Alumni Research Foundation.

\subsection*{Notations}

 Throughout the text, we let $|u - v|$ denote the Euclidean distance between any two points $u, v \in \mathbb{R}^2$; we also let $\dist (\mathcal{S}, \mathcal{S}') = \inf_{s \in S} \inf_{s' \in S'} |s - s'|$ denote the distance between any sets $\mathcal{S}, \mathcal{S}' \subset \mathbb{R}^2$. We adopt a convention to write $C>1$ for constants which are large and $C>0$ for constants which are small. In general, these constants change from statement to statement.

\section{Setup and main results}

\label{Section_statement}

\subsection{Tilings and height functions}	
\label{Section_height_function}

Let us introduce the coordinate system on the triangular grid. We use $x$ and $y$ coordinate axes, which are inclined to each other by 120 degrees: $y$--coordinate grows up and $x$--coordinate grows in down-right direction, as in the left panel of Figure \ref{Figure_hexagon}. The \emph{triangular lattice} $\mathbb{T}$ in this coordinate system becomes a graph whose vertex set is $\mathbb{Z}^2$ and whose edge set consists of edges connecting $(x, y), (x', y') \in \mathbb{Z}^2$ if and only if $(x' - x, y' - y) \in \big\{ (1, 0), (-1, 0), (0, 1), (0, -1), (1, 1), (-1, -1) \big\}$. The faces of $\mathbb{T}$ are therefore triangles with vertices of the form $\big\{ (x, y), (x + 1, y), (x + 1, y + 1) \big\}$ or $\big\{ (x, y), (x, y + 1), (x + 1, y + 1) \big\}$. A \emph{domain} $R \subseteq \mathbb{T}$ is a simply-connected\footnote{Our methods should also extend beyond the simply-connected setting, but we do not address this here.} induced subgraph of $\mathbb{T}$. The \emph{boundary} $\partial R \subseteq R$ is the set of all vertices $v \in R$ adjacent to a vertex in $\mathbb{T} \setminus R$.
	

A pair of adjacent triangular faces of $\mathbb T$ forms a parallelogram, which we refer to as a \emph{lozenge} or \emph{tile}. Lozenges can be oriented in one of three ways; see the left panel of \Cref{Figure_height_function} for all three orientations. We refer to the topmost lozenge there (that is, one with vertices of the form $\big\{ (x, y), (x, y + 1), (x + 1, y + 2), (x + 1, y + 1) \big\}$) as a \emph{type $1$} lozenge. Similarly, we refer to the middle (with vertices of the form $\big\{ (x, y), (x + 1, y), (x + 2, y + 1), (x + 1, y + 1) \big\}$) and bottom (vertices of the form $\big\{ (x, y), (x, y + 1), (x + 1, y + 1), (x + 1, y) \big\}$) ones there as \emph{type $2$} and \emph{type $3$} lozenges, respectively. The horizontal lozenges discussed in the introduction are of type $2$. We deal with tilings of $R$ by lozenges of types $1$, $2$, and $3$. Let $\Omega (R)$ denote the set of all tilings of $R$; if $\Omega (R)$ is nonempty, we say that $R$ is \emph{tileable}.
	
	Associated with any tiling of $R$ is a \emph{height function}\footnote{There are other definitions of the height function; see, e.g., \cite[Section 1.4]{Gorin_book}  for a more symmetric one. All definitions contain the same information, because they differ from each other by explicit affine transformations.}  $H: R \rightarrow \mathbb{Z}$, namely, a function on the vertices of $R$ that satisfies
	\begin{flalign*}
		f(v) - f (u) \in \{ 0, 1 \}, \quad \text{whenever $u = (x, y)$ and $v \in \big\{ (x + 1, y), (x, y + 1), (x + 1, y + 1) \big\}$},
	\end{flalign*}
	
	\noindent for some $(x, y) \in \mathbb{Z}^2$. We refer to the restriction $h = H|_{\partial R}$ as a \emph{boundary height function}.
	
	For a fixed vertex $v \in R$ and integer $m \in \mathbb{Z}$, one can associate with any tiling of $R$ a height function $H: R \rightarrow \mathbb{Z}$ as follows. First, set $H (v) = m$, and then define $H$ at the remaining vertices of $R$ in such a way that the height functions along the four vertices of any lozenge in the tiling are of the form depicted on the left panel of \Cref{Figure_height_function}. In particular, we require that $H (x + 1, y) - H (x, y) = 1$ if and only if $(x, y)$ and $(x + 1, y)$ are vertices of the same type $1$ lozenge, and that $H (x, y + 1) - H (x, y) = 1$ if and only if $(x, y)$ and $(x, y + 1)$ are vertices of the same type $2$ lozenge. Since $R$ is simply-connected, a height function on $R$ is uniquely determined by these conditions and the value of $H(v) = m$.
	
	 We refer to the right panel of \Cref{Figure_height_function} for an example of a height function; there we chose $v$ to be the bottom--left corner of the hexagon and $m=0$. We can also view a lozenge tiling as a projection of a stepped surface in three-dimensional space: the three types of lozenges then become three types of faces of a unit cube, as shown in the middle panel of \Cref{Figure_height_function}. The height function then counts the (signed) distance between a given point in three-dimensional space and coordinate plane corresponding to type $3$ lozenges. In particular, the group of vertices with height $0$ around the left corner of the right panel of Figure \ref{Figure_height_function} lies in this plane.
 	
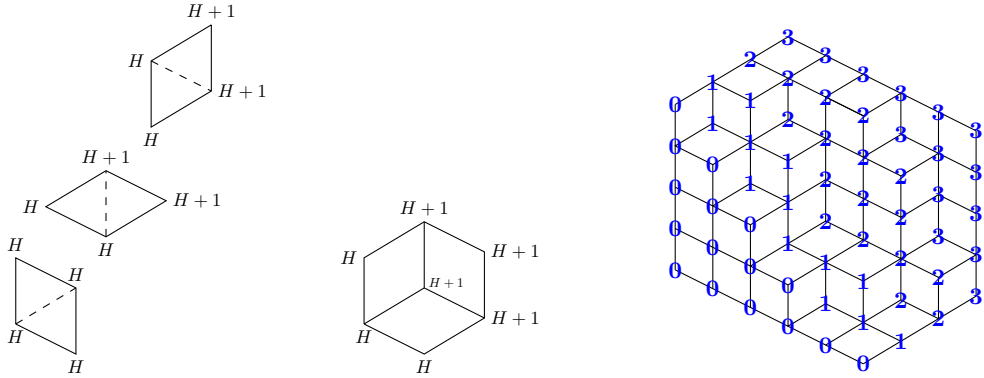
\begin{figure}
\begin{center}
\begin{tikzpicture}[
>=stealth,
auto,
style={
scale = .4,
cm={1,-0.5,0,1.1,(0,0)}
}
]
\draw[-, black] (17.5, 9) node[above, scale = .7]{$H + 1$}-- (17.5, 7) node[right, scale = .7]{$H + 1$} -- (15.5, 5) node[below, scale = .7]{$H$} -- (15.5, 7) node[left, scale = .7]{$H$} -- (17.5, 9);
\draw[-, dashed, black] (15.5, 7) -- (17.5, 7);

\draw[-, black] (14, 3) node[above, scale = .7]{$H + 1$} -- (16, 3) node[right, scale = .7]{$H + 1$} -- (14, 1) node[below, scale = .7]{$H$} -- (12, 1) node[left, scale = .7]{$H$}-- (14, 3);
\draw[-, dashed, black] (14, 1) -- (14, 3);

\draw[-, black] (11, -3) node[below, scale = .7]{$H$} -- (11, -1) node[above, scale = .7]{$H$} -- (13, -1) node[above, scale = .7]{$H$} -- (13, -3) node[below, scale = .7]{$H$} -- (11, -3);
\draw[-, dashed, black] (11, -3) -- (13, -1);

\end{tikzpicture}
\qquad
\begin{tikzpicture}[
>=stealth,
auto,
style={
scale = .4,
cm={1,-0.5,0,1.1,(0,0)}
}
]

\draw[-, black] (-9, 1.5) node[below, scale = .7]{$H$} -- (-7, 1.5) node[below, scale = .7]{$H$} -- (-5, 3.5) node[right, scale = .7]{$H + 1$} -- (-5, 5.5) node[right, scale = .7]{$H + 1$} -- (-7, 5.5) node[above, scale = .7]{$H + 1$} -- (-9, 3.5) node[left, scale = .7]{$H$} -- (-9, 1.5);
\draw[-, black] (-9, 1.5) -- (-7, 3.5) node[above=2,right, scale = .5]{$H+1$} -- (-7, 5.5);
\draw[-, black] (-7, 3.5) -- (-5, 3.5);
\end{tikzpicture}
\qquad\qquad
\begin{tikzpicture}[
>=stealth,
auto,
style={
scale = .5,
cm={1,-0.5,0,1.1,(0,0)}
}
]
\draw[-, black] (0, 0) -- (5, 0);
\draw[-, black] (0, 0) -- (0, 4);
\draw[-, black] (5, 0) -- (8, 3);
\draw[-, black] (0, 4) -- (3, 7);
\draw[-, black] (8, 3) -- (8, 7);
\draw[-, black] (8, 7) -- (3, 7);

\draw[-, black] (1, 0) -- (1, 3) -- (0, 3) -- (1, 4) -- (1, 5) -- (2, 5) -- (4, 7);
\draw[-, black] (0, 2) -- (2, 2) -- (2, 0);
\draw[-, black] (0, 1) -- (3, 1) -- (3, 0) -- (4, 1) -- (6, 1);
\draw[-, black] (2, 6) -- (5, 6) -- (6, 7) -- (6, 6) -- (8, 6);
\draw[-, black] (4, 6) -- (5, 7);
\draw[-, black] (1, 4) -- (2, 4) -- (2, 5);
\draw[-, black] (2, 4) -- (3, 5) -- (3, 6);
\draw[-, black] (1, 3) -- (2, 4);
\draw[-, black] (2, 2) -- (3, 3) -- (3, 4) -- (2, 4) -- (2, 3) -- (1, 2);
\draw[-, black] (2, 3) -- (3, 3) -- (3, 2) -- (4, 3) -- (4, 6) -- (5, 6) -- (5, 3) -- (4, 3);
\draw[-, black] (3, 5) -- (5, 5) -- (6, 6);
\draw[-, black] (7, 7) -- (7, 4) -- (8, 4) -- (7, 3) -- (7, 2) -- (6, 2) -- (4, 0);
\draw[-, black] (8, 5) -- (7, 5) -- (6, 4) -- (6, 5) -- (7, 6);
\draw[-, black] (3, 3) -- (4, 4) -- (6, 4);
\draw[-, black] (3, 4) -- (4, 5);
\draw[-, black] (6, 5) -- (5, 5);
\draw[-, black] (7, 4) -- (6, 3) -- (7, 3);
\draw[-, black] (2, 1) -- (3, 2) -- (5, 2) -- (6, 3);
\draw[-, black] (3, 1) -- (5, 3) -- (6, 3) -- (6, 4);
\draw[-, black] (4, 1) -- (4, 2);
\draw[-, black] (5, 1) -- (5, 2);
\draw[-, black] (6, 2) -- (6, 3);

\draw[-,blue] (0, 0) circle [radius = 0] node[scale = .9]{$\mathbf{0}$};
\draw[-,blue] (1, 0) circle [radius = 0] node[scale = .9]{$\mathbf{0}$};
\draw[-,blue] (2, 0) circle [radius = 0] node[scale = .9]{$\mathbf{0}$};
\draw[-,blue] (3, 0) circle [radius = 0] node[scale = .9]{$\mathbf{0}$};
\draw[-,blue] (4, 0) circle [radius = 0] node[scale = .9]{$\mathbf{0}$};
\draw[-,blue] (5, 0) circle [radius = 0] node[scale = .9]{$\mathbf{0}$};

\draw[-,blue] (0, 1) circle [radius = 0] node[scale = .9]{$\mathbf{0}$};
\draw[-,blue] (1, 1) circle [radius = 0] node[scale = .9]{$\mathbf{0}$};
\draw[-,blue] (2, 1) circle [radius = 0] node[scale = .9]{$\mathbf{0}$};
\draw[-,blue] (3, 1) circle [radius = 0] node[scale = .9]{$\mathbf{0}$};
\draw[-,blue] (4, 1) circle [radius = 0] node[scale = .9]{$\mathbf{1}$};
\draw[-,blue] (5, 1) circle [radius = 0] node[scale = .9]{$\mathbf{1}$};
\draw[-,blue] (6, 1) circle [radius = 0] node[scale = .9]{$\mathbf{1}$};

\draw[-,blue] (0, 2) circle [radius = 0] node[scale = .9]{$\mathbf{0}$};
\draw[-,blue] (1, 2) circle [radius = 0] node[scale = .9]{$\mathbf{0}$};
\draw[-,blue] (2, 2) circle [radius = 0] node[scale = .9]{$\mathbf{0}$};
\draw[-,blue] (3, 2) circle [radius = 0] node[scale = .9]{$\mathbf{1}$};
\draw[-,blue] (4, 2) circle [radius = 0] node[scale = .9]{$\mathbf{1}$};
\draw[-,blue] (5, 2) circle [radius = 0] node[scale = .9]{$\mathbf{1}$};
\draw[-,blue] (6, 2) circle [radius = 0] node[scale = .9]{$\mathbf{2}$};
\draw[-,blue] (7, 2) circle [radius = 0] node[scale = .9]{$\mathbf{2}$};

\draw[-,blue] (0, 3) circle [radius = 0] node[scale = .9]{$\mathbf{0}$};
\draw[-,blue] (1, 3) circle [radius = 0] node[scale = .9]{$\mathbf{0}$};
\draw[-,blue] (2, 3) circle [radius = 0] node[scale = .9]{$\mathbf{1}$};
\draw[-,blue] (3, 3) circle [radius = 0] node[scale = .9]{$\mathbf{1}$};
\draw[-,blue] (4, 3) circle [radius = 0] node[scale = .9]{$\mathbf{2}$};
\draw[-,blue] (5, 3) circle [radius = 0] node[scale = .9]{$\mathbf{2}$};
\draw[-,blue] (6, 3) circle [radius = 0] node[scale = .9]{$\mathbf{2}$};
\draw[-,blue] (7, 3) circle [radius = 0] node[scale = .9]{$\mathbf{2}$};
\draw[-,blue] (8, 3) circle [radius = 0] node[scale = .9]{$\mathbf{3}$};

\draw[-,blue] (0, 4) circle [radius = 0] node[scale = .9]{$\mathbf{0}$};
\draw[-,blue] (1, 4) circle [radius = 0] node[scale = .9]{$\mathbf{1}$};
\draw[-,blue] (2, 4) circle [radius = 0] node[scale = .9]{$\mathbf{1}$};
\draw[-,blue] (3, 4) circle [radius = 0] node[scale = .9]{$\mathbf{1}$};
\draw[-,blue] (4, 4) circle [radius = 0] node[scale = .9]{$\mathbf{2}$};
\draw[-,blue] (5, 4) circle [radius = 0] node[scale = .9]{$\mathbf{2}$};
\draw[-,blue] (6, 4) circle [radius = 0] node[scale = .9]{$\mathbf{2}$};
\draw[-,blue] (7, 4) circle [radius = 0] node[scale = .9]{$\mathbf{3}$};
\draw[-,blue] (8, 4) circle [radius = 0] node[scale = .9]{$\mathbf{3}$};

\draw[-,blue] (1, 5) circle [radius = 0] node[scale = .9]{$\mathbf{1}$};
\draw[-,blue] (2, 5) circle [radius = 0] node[scale = .9]{$\mathbf{1}$};
\draw[-,blue] (3, 5) circle [radius = 0] node[scale = .9]{$\mathbf{2}$};
\draw[-,blue] (4, 5) circle [radius = 0] node[scale = .9]{$\mathbf{2}$};
\draw[-,blue] (5, 5) circle [radius = 0] node[scale = .9]{$\mathbf{2}$};
\draw[-,blue] (6, 5) circle [radius = 0] node[scale = .9]{$\mathbf{2}$};
\draw[-,blue] (7, 5) circle [radius = 0] node[scale = .9]{$\mathbf{3}$};
\draw[-,blue] (8, 5) circle [radius = 0] node[scale = .9]{$\mathbf{3}$};

\draw[-,blue] (2, 6) circle [radius = 0] node[scale = .9]{$\mathbf{2}$};
\draw[-,blue] (3, 6) circle [radius = 0] node[scale = .9]{$\mathbf{2}$};
\draw[-,blue] (4, 6) circle [radius = 0] node[scale = .9]{$\mathbf{2}$};
\draw[-,blue] (5, 6) circle [radius = 0] node[scale = .9]{$\mathbf{2}$};
\draw[-,blue] (6, 6) circle [radius = 0] node[scale = .9]{$\mathbf{3}$};
\draw[-,blue] (7, 6) circle [radius = 0] node[scale = .9]{$\mathbf{3}$};
\draw[-,blue] (8, 6) circle [radius = 0] node[scale = .9]{$\mathbf{3}$};

\draw[-,blue] (3, 7) circle [radius = 0] node[scale = .9]{$\mathbf{3}$};
\draw[-,blue] (4, 7) circle [radius = 0] node[scale = .9]{$\mathbf{3}$};
\draw[-,blue] (5, 7) circle [radius = 0] node[scale = .9]{$\mathbf{3}$};
\draw[-,blue] (6, 7) circle [radius = 0] node[scale = .9]{$\mathbf{3}$};
\draw[-,blue] (7, 7) circle [radius = 0] node[scale = .9]{$\mathbf{3}$};
\draw[-,blue] (8, 7) circle [radius = 0] node[scale = .9]{$\mathbf{3}$};
\end{tikzpicture}

\end{center}

\caption{\label{Figure_height_function} Left panel: three types of lozenges and corresponding changes of the height function. Middle panel: a tiled unit hexagon can be viewed as a projection of three faces of a unit cube in 3D. Right panel: a lozenge tiling and corresponding values of the height function.}
\end{figure}

One important observation is that, if there exists a tiling $\mathscr{M}$ of $R$ associated with some height function $H$, then the boundary height function $h = H |_{\partial R}$ is independent of $\mathscr{M}$ and is uniquely determined by $R$ (except for a global shift, which was above fixed by the value of $H(v) = m$).

	It will be useful to introduce continuum analogs of the above notions. So, set ${\mathcal{T} = \bigl\{ (s, t): 0 < s + t < 1 \bigr\} \subset \mathbb{R}^2}$ and its closure $\overline{\mathcal{T}} = \big\{ (s, t) : 0 \le s + t \le 1 \big\}$. We interpret $\overline{\mathcal{T}}$ as the set of possible gradients, also called \emph{slopes}, for a continuum height function; $\mathcal{T}$ is then the set of \emph{liquid} slopes, whose associated tilings still ``appear random.'' For any simply-connected domain $\mathfrak{R} \subset \mathbb{R}^2$, we say that a function $\mathfrak{H} : \mathfrak{R} \rightarrow \mathbb{R}$ is \emph{admissible} if $\mathfrak{H}$ is $1$-Lipschitz and $\nabla \mathfrak{H} (z) \in \overline{\mathcal{T}}$ for almost all $z \in \mathfrak{R}$. We further say a function $\mathfrak{h}: \partial \mathfrak{R} \rightarrow \mathbb{R}$ \emph{admits an admissible extension to $\mathfrak{R}$} if $\Adm (\mathfrak{R}; \mathfrak{h})$, the set of admissible functions $\mathfrak{H}: \mathfrak{R} \rightarrow \mathbb{R}$ with $\mathfrak{H} |_{\partial \mathfrak{R}} = \mathfrak{h}$, is not empty.

\subsection{Approximation of domains}

Let us now explain the sense in which a sequence of discrete domains ``converges'' to a continuum one. To that end, for any subset $\mathcal{S} \subseteq \mathbb{R}^2$ and points $x_1, x_2 \in \mathcal{S}$, let $d_{\mathcal{S}} (x_1, x_2) = \inf_{\gamma} |\gamma|$, where $\gamma \subseteq \mathcal{S}$ is taken over all paths in $\mathcal{S}$ connecting $x_1$ and $x_2$, and $|\gamma|$ denotes the length of $\gamma$. Next, we say that a sequence of subsets \emph{$\mathfrak{R}_1, \mathfrak{R}_2, \ldots \subset \mathbb{R}^2$ converges to $\mathfrak{R} \subset \mathbb{R}^2$}, and write $\lim_{N \rightarrow \infty} \mathfrak{R}_N = \mathfrak{R}$, if for any $\delta > 0$ there exists an integer $N_0 = N_0 (\delta) > 1$ such that the following two properties hold whenever $N > N_0$. First, for any $x \in \mathfrak{R}$, there exist $x_N \in \mathfrak{R}_N$ and $x' \in \mathfrak{R} \cap \mathfrak{R}_N$ such that
	\begin{flalign}
		\label{dxxnrrn}
		\displaystyle\max \big\{ d_{\mathfrak{R}} (x, x'), d_{\mathfrak{R}_N} (x_N, x') \big\} < \delta.
	\end{flalign}
	
	\noindent Second, for any $x_N \in \mathfrak{R}_N$, there exist $x \in \mathfrak{R}$ and $x' \in \mathfrak{R} \cap \mathfrak{R}_N$ such that \eqref{dxxnrrn} again holds.
	
	In this case, we moreover say that a sequence of functions \emph{$\mathfrak{h}_N: \partial \mathfrak{R}_N \rightarrow \mathbb{R}$ converges to $\mathfrak{h}: \partial \mathfrak{R} \rightarrow \mathbb{R}$}, and write $\lim_{N \rightarrow \infty} \mathfrak{h}_N = \mathfrak{h}$, if for every real number $\delta > 0$ there exists an integer $N_0 = N_0 (\delta) > 1$ such that the following two properties hold whenever $N > N_0$. First, for each $x \in \partial \mathfrak{R}$, there exist $x_N \in \partial \mathfrak{R}_N$ and $x' \in  \mathfrak{R} \cap \mathfrak{R}_N$ such that \eqref{dxxnrrn} and $\big| \mathfrak{h}_N (x_N) - \mathfrak{h} (x) \big| < \delta$ both hold. Second, for each $x_N \in \partial \mathfrak{R}_N$, there exist $x \in \partial \mathfrak{R}$ and $x' \in \mathfrak{R} \cap \mathfrak{R}_N$ such that the same inequalities are satisfied.

	
\begin{definition}\label{Def_domain_convergence} Suppose we are given a sequence $R_1, R_2, \ldots \subset \mathbb{T}$ of domains and boundary height functions $h_N : \partial R_N \rightarrow \mathbb{Z}$, $N=1,2,\dots$. Denote for each integer $N \ge 1$ the domain $\mathfrak{R}_N = \tfrac{1}{N} R_N$ and function $\mathfrak{h}_N : \partial \mathfrak{R}_N \rightarrow \mathbb{R}$ by $\mathfrak{h}_N (x) = \tfrac{1}{N} h_N (\tfrac{1}{N} x)$.
 We say that $(R_N; h_N)$ converges to $(\mathfrak{R}; \mathfrak{h})$, for some $\mathfrak{R} \subset \mathbb{R}^2$ and $\mathfrak{h} : \partial \mathfrak{R} \rightarrow \mathbb{R}$, if $\lim_{N \rightarrow \infty} \mathfrak{R}_N = \mathfrak{R}$ and $\lim_{N \rightarrow \infty} \mathfrak{h}_N = \mathfrak{h}$.
\end{definition}

\begin{example} \label{Example_domains}
 Let $\mathfrak R$ be a tileable polygon drawn on $\mathbb T$. Taking any tiling of $\mathfrak R$ and using definitions of Section \ref{Section_height_function}, we get the corresponding boundary height function $\mathfrak h: \partial \mathfrak{R} \rightarrow \mathbb{Z}$. For $N=1,2,\dots$, let $R_N=N\cdot R$. It is straightforward to check that $R_N$ is tileable; let $h_N$ be the corresponding boundary height function. In this situation $(R_N; h_N)$ converges to $(\mathfrak{R}; \mathfrak{h})$, because we can choose $x=x'=x_N$ in \eqref{dxxnrrn}.

 Another point of view on this example is that a polygon $\mathfrak R$ is fixed and we are tiling it with lozenges of smaller and smaller side lengths $1/N$.
\end{example}

\subsection{Convergence to the GUE--corners: statements}

We are now ready to state our main theorem about convergence of random tilings to the GUE--corners process. We first explain a simpler version of the theorem, and then proceed to the most general setting.

Throughout this section we fix a sequence of tileable domains $R_N\subset \mathbb{T}$ with corresponding boundary height functions $h_N$, $N=1,2,\dots$ and assume that they converge to $(\mathfrak R;\mathfrak h)$. We are going to choose a specific straight segment in the boundary of each $R_N$ and assume that there are two adjacent straight segments inclined to it under an angle of 120 degrees. There are six options depending on the orientation of the straight segment, as shown in Figure \ref{Figure_six}. In each of these six situations we trace a particular type of lozenges (whose diagonal is parallel to the straight segment of interest) and get an interlacing configuration of lozenges, in the same way as in Figure \ref{Figure_hexagon} we were tracing horizontal lozenges near a vertical segment of the boundary. We investigate the coordinates of these lozenges counted in the direction parallel to the straight segment of interest; in this way in Figure \ref{Figure_hexagon} the segment was vertical and, hence, we were studying the vertical coordinates $y_i^k$. For notational simplicity, we state the next theorem for the vertical boundaries, as in the top-left panel of Figure \ref{Figure_six}; applying rotations, exactly the same statement holds for other five orientations.

\begin{figure}[t]
\center
\includegraphics[width=0.7\textwidth]{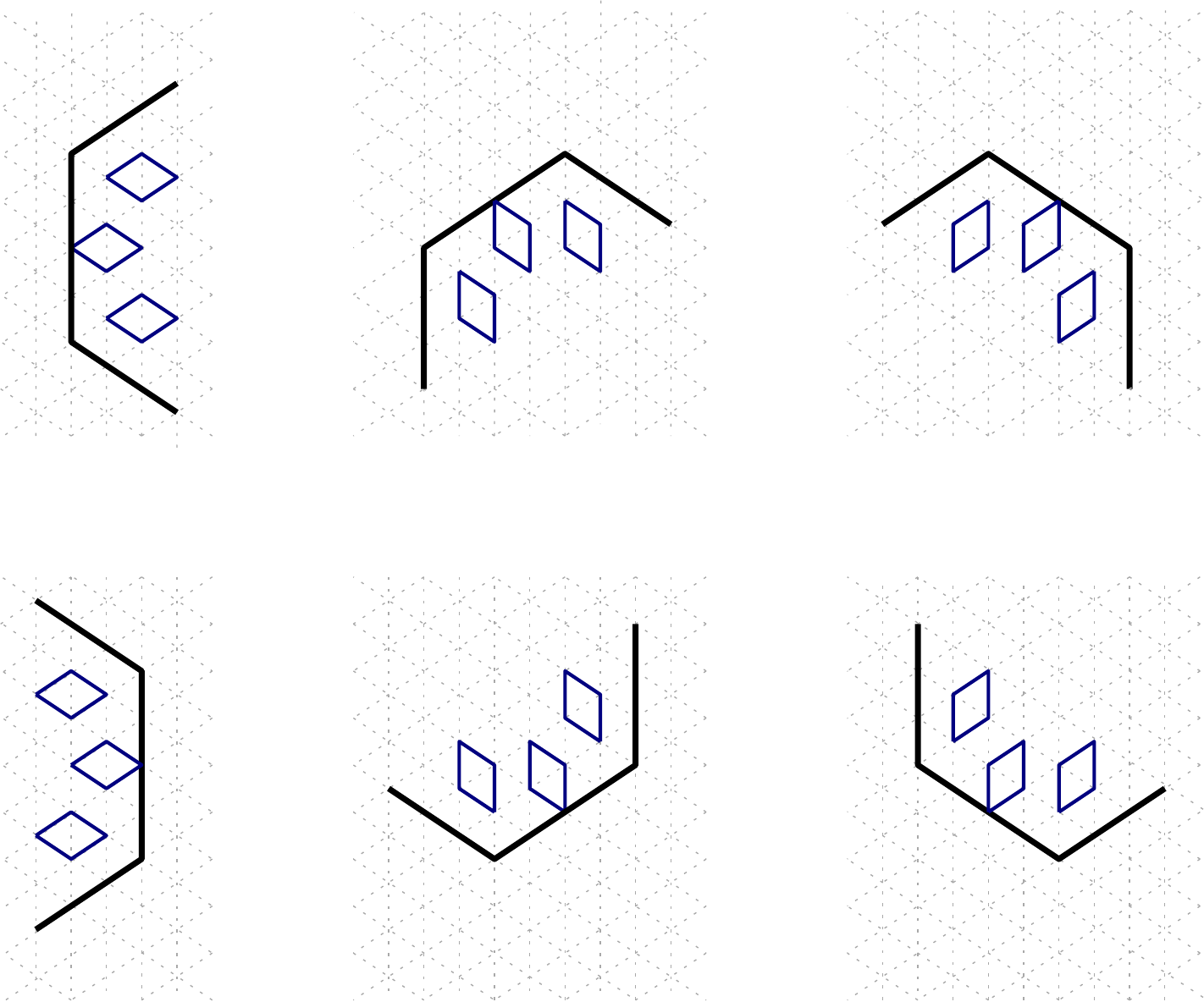}
\caption{Six possible orientations of a straight segment and corresponding types of lozenges that we trace.}
\label{Figure_six}
\end{figure}

\begin{theorem} \label{Theorem_GUE_simple}
 Take a sequence of tileable simply-connected domains $R_N\subset \mathbb R$ and corresponding boundary height functions $h_N$. Suppose that $(R_N,h_N)$ converges to some $(\mathfrak R, \mathfrak h)$, where  $\mathfrak{R} \subset \mathbb{R}^2$  is a simply-connected domain with piecewise smooth boundary. Further, suppose that the boundary of each $R_N$ has a distinguished vertical segment $I_N$ with two adjacent segments $I_N^{(l)}$, $I_N^{(r)}$ inclined by 120 degrees to $I_N$, as in the top-left panel of Figure \ref{Figure_six}. Assume that as $N\to\infty$, rescaled by $N$, the positions and lengths of $I_N$, $I_N^{(l)}$, and $I_N^{(r)}$ converge to positions and lengths (which must remain positive) of straight segments of the boundary of $\mathfrak R$. Let $\{y_i^k(N)\}_{1\le i \le k}$ denote the interlacing array of the vertical coordinates of horizontal lozenges adjacent to $I_N$ in uniformly random tiling of $R_N$. Then there exist constants $m_N\in\mathbb R$ and $\sigma\ge 0$, such that
 \begin{equation}
 \lim_{N\to\infty} \left( \frac{y_i^k(N)- m_N}{\sqrt{N}}\right)_{1\le i \le k} = \bigl(\sigma \xi_i^k\bigr)_{1\le i \le k},
\end{equation}
in the sense of convergence of finite-dimensional distributions (jointly over all $i$ and $k$ such that $1\le i \le k$), and where $\{\xi_i^k\}$ is the GUE--corners process of Definition \ref{Def_GUE_corners}.
\end{theorem}
\begin{remark}
 There are two ways to think about the constants $m_N$ and $\sigma$. First, we can set ${m_N=\E y_1^1(N)}$ and $\sigma=\lim_{N\to\infty} N^{-1/2} \sqrt{\Var(y_1^1(N))}$, i.e. they are the mean and asymptotic standard deviation of $y_1^1$. Second, $\sigma$ and the leading order\footnote{Subleading terms, i.e.\ $m_N-N\cdot \left[\lim_{N\to\infty}\tfrac{m_N}{N}\right]$ might depend on the exact way $(R_N,h_N)$ approximates $(\mathfrak R,\mathfrak h)$, and there is no way to reconstruct them only from  $(\mathfrak R,\mathfrak h)$ and the corresponding tiling limit shape.}
  behavior of $m_N$ can be reconstructed through the limit shape (Law of Large Numbers) for tilings, see Lemma \ref{Lemma_GUE_trap} below for the formulas. We also mention that there might be degenerate situations, such as frozen domains, in which $\sigma=0$.
\end{remark}

\begin{figure}[t]
\center
\includegraphics[width=0.45\textwidth]{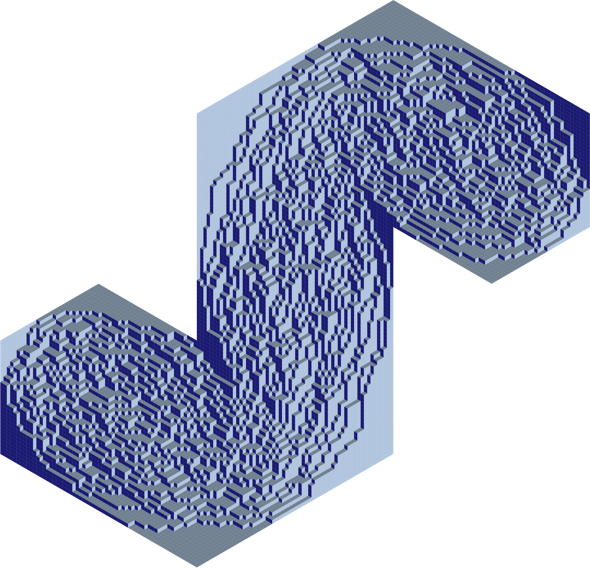}\qquad
\includegraphics[width=0.49\textwidth]{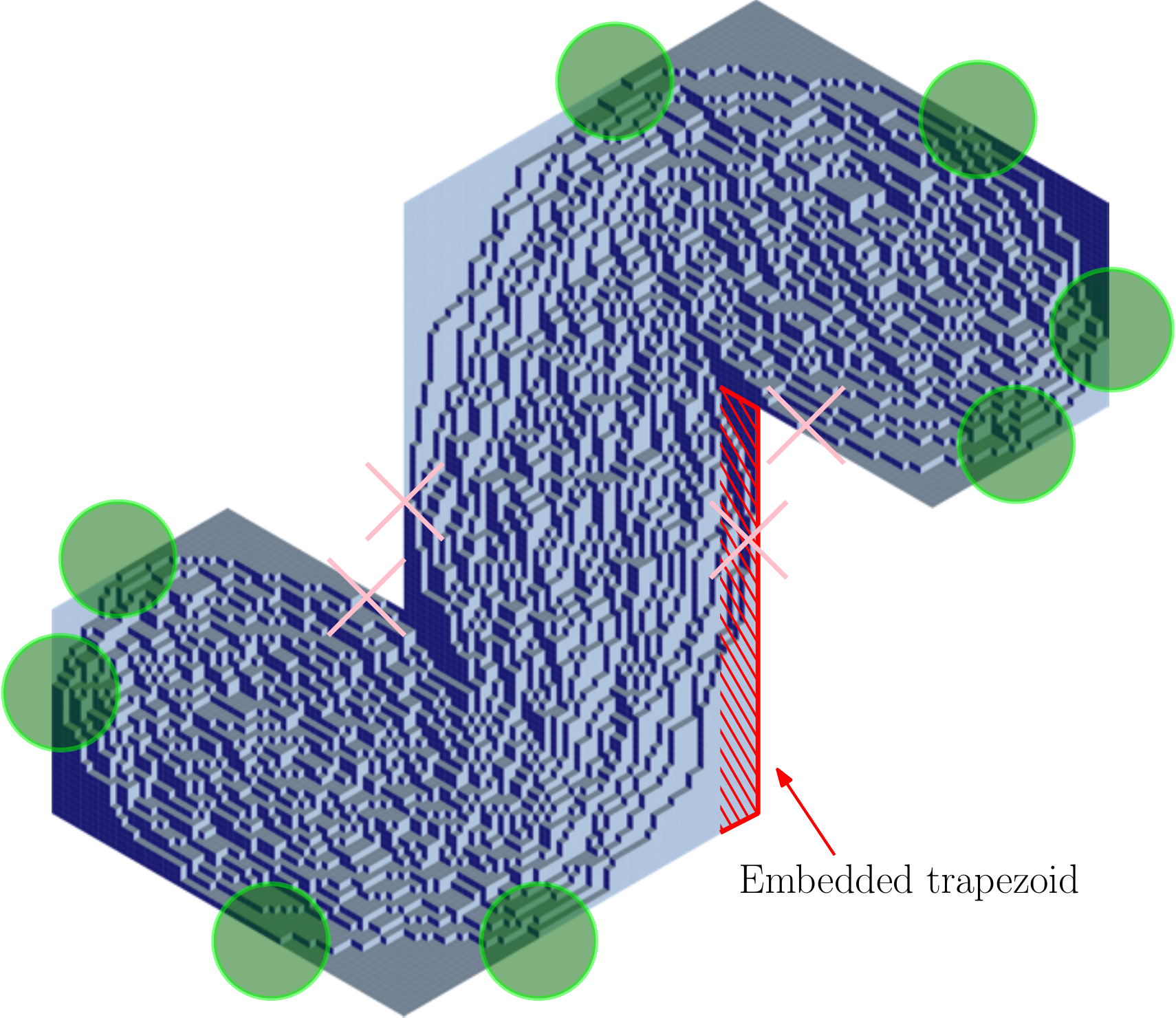}
\caption{Left panel: Sample of a uniformly random lozenge tiling of a large S-shaped domain. Three types of tiles are shown in three colors. Right panel: Theorem \ref{Theorem_GUE_simple} covers convergence to eight instances of the GUE--corners process shown in green circles. It does not cover four other instances shown by pink crosses. One embedded trapezoid is shown in red. The simulations are taken from \cite{KS} and we are grateful to the authors of that paper for allowing us to use their software.
\label{Figure_GUE_cases}}
\end{figure}

Figure \ref{Figure_GUE_cases} illustrates Theorem \ref{Theorem_GUE_simple} by showing the situations in which it applies. The same figure contains several situations where Theorem \ref{Theorem_GUE_simple} does not apply because there are no three straight segments of the boundary inclined by 120 degrees to each other; yet, we believe that the GUE--corners process still appears asymptotically in those situations as well. The reason is that the system creates the straight boundaries on its own inside the frozen regions containing with high probability only one type of lozenges; once we know that an auxiliary straight boundary is there, we can apply a version of Theorem \ref{Theorem_GUE_simple} again. Let us state this version.

\begin{definition}
 Take a lozenge tiling $\mathfrak L\in \Omega (R)$ of a domain $R\subset \mathbb T$. We say that a trapezoid with straight sides $I^{(l)}$, $I$, $I^{(r)}$ is embedded in $\mathfrak L$, if:
 \begin{itemize}
   \item  $I^{(l)}$, $I$, $I^{(r)}$ are three segments drawn on the triangular grid and inclined to each other by 120 degrees in one of the six configurations of Figure \ref{Figure_six}; and
   \item  $I^{(l)}$, $I$, and $I^{(r)}$ are inside (or on the boundary) of $R$; and
   \item No lozenges of $\mathfrak L$ intersect  $I^{(l)}$, $I$, and $I^{(r)}$. In other words, these segments are formed by unions of various sides of lozenges of the tiling.
 \end{itemize}
\end{definition}
We refer to the right panel of Figure \ref{Figure_GUE_cases} for an example of the embedded trapezoid.

\begin{theorem} \label{Theorem_GUE_general}
 Take a sequence of tileable simply-connected domains $R_N\subset \mathbb R$ and corresponding boundary height functions $h_N$. Suppose that $(R_N,h_N)$ converges to some $(\mathfrak R, \mathfrak h)$, where  $\mathfrak{R} \subset \mathbb{R}^2$  is a simply-connected domain with piecewise smooth boundary. Further, suppose that with probability tending to $1$ as $N\to\infty$, each tiling of  $R_N$ has an embedded trapezoid (which depends on $N$, but not on the tiling) with straight sides $I^{(l)}_N$, $I_N$, $I^{(r)}_N$, oriented as in top-left panel of Figure \ref{Figure_six}; in particular, $I_N$ is vertical. Assume that as $N\to\infty$, rescaled by $N$, the lengths of $I_N$, $I_N^{(l)}$, and $I_N^{(r)}$ converge to positive constants. Let $\{y_i^k(N)\}_{1\le i \le k}$ denote the interlacing array of the vertical coordinates of horizontal lozenges adjacent to $I_N$ in uniformly random tiling of $R_N$. Then there exist constants $m_N\in\mathbb R$ and $\sigma\ge 0$, such that
 \begin{equation}
\label{eq_gen_to_GUE}
 \lim_{N\to\infty} \left( \frac{y_i^k(N)- m_N}{\sqrt{N}}\right)_{1\le i \le k} = \bigl(\sigma \xi_i^k\bigr)_{1\le i \le k},
\end{equation}
in the sense of convergence of finite-dimensional distributions (jointly over all $i$ and $k$ such that $1\le i \le k$), and where $\{\xi_i^k\}$ is the GUE--corners process of Definition \ref{Def_GUE_corners}.
\end{theorem}
\begin{remark}
 We stated Theorem \ref{Theorem_GUE_general} for the orientation of the top-left panel of Figure \ref{Figure_six}. Applying rotations, we conclude that it also applies for five other orientations.
\end{remark}

Theorem \ref{Theorem_GUE_simple} is a particular case of Theorem \ref{Theorem_GUE_general}: indeed, in this situation the embedded trapezoid is formed by three straight segments of the boundary of the domain. This trapezoid is embedded in each tiling of $R_N$ with probability $1$, because lozenges are never allowed to cross the boundary of the domain.

The conditions of Theorem \ref{Theorem_GUE_general} are not only sufficient, but, in a sense, they are almost necessary: the only part which can be potentially weakened is the linear (in $N$) growth of the lengths of $I^{(l)}_N$, $I_N$, and $I^{(r)}_N$; there might be situations in which these lengths grow sublinearly and convergence to the GUE--corners process still holds. Let us explain necessity of the existence with high probability of the embedded trapezoid.

Imagine that we expect convergence of lozenge tilings to the GUE--corners process in a certain situation (still near a straight segment of a boundary of a domain). The eigenvalues $\xi_i^k$ from Definition \ref{Def_GUE_corners} almost surely interlace, see e.g., \cite[Lecture 20]{Gorin_book}, which means that
$$
 \xi_i^k < \xi_i^{k+1}< \xi_{i+1}^{k+1},\quad 1\le i \le k.
$$
Hence, convergence to the GUE--corners process implies the existence of a similarly interlacing configuration of lozenges. Combinatorially, once we see an interlacing configuration of horizontal lozenges near a vertical segment of the boundary, far up and down along this boundary we should observe only lozenges of types 1 and 3, respectively, cf.\ the left panel of Figure \ref{Figure_e_trap}. In turn, the presence of lozenges of only one type far up and far down, readily leads to the appearance of an embedded trapezoid.

\begin{figure}[t]
\center
\includegraphics[width=0.3\textwidth]{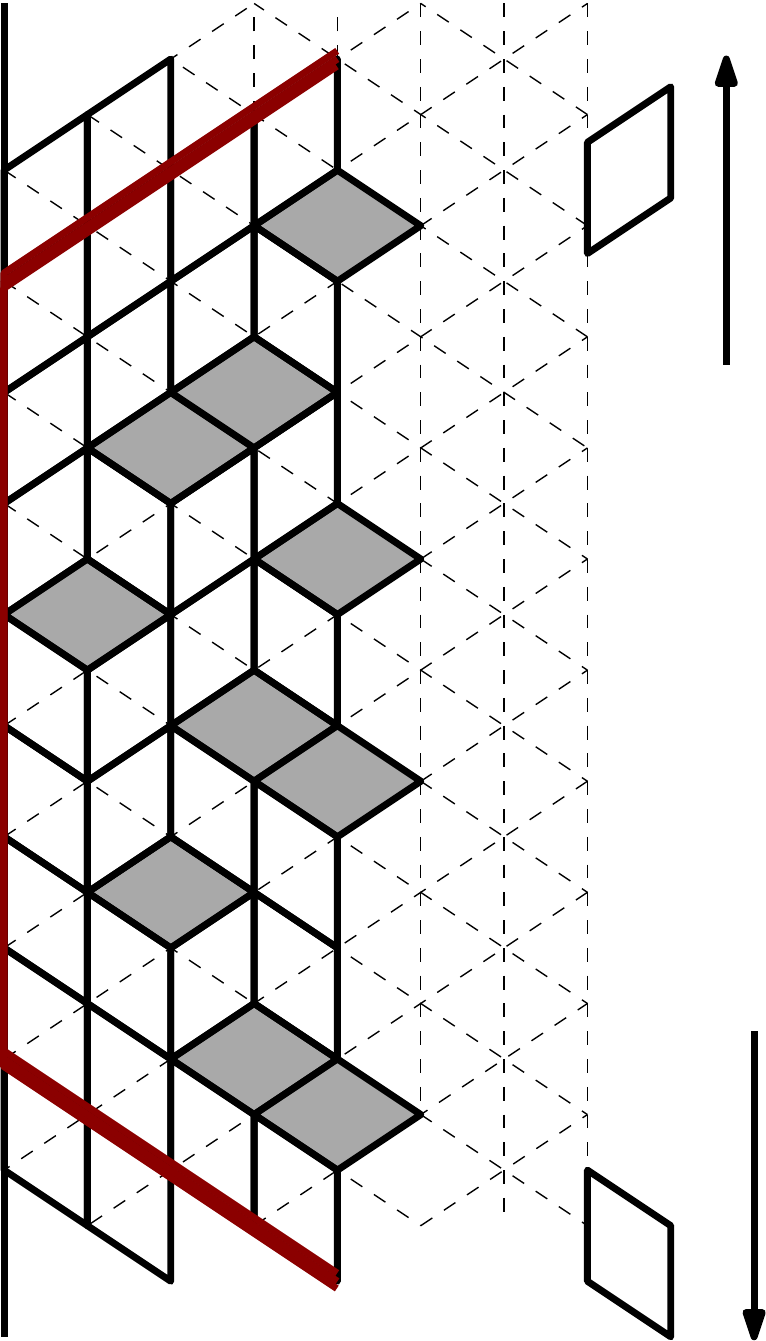}
 \hskip4cm
\includegraphics[width=0.25\textwidth]{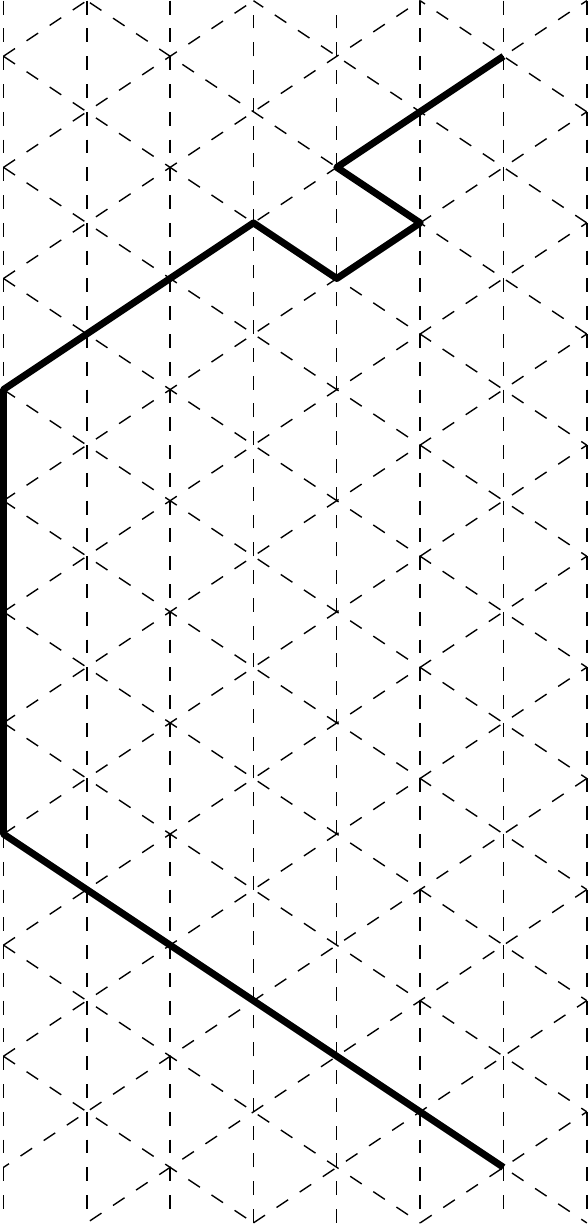}
\caption{Left panel: Interlacing configuration of horizontal lozenges near a vertical segment of the boundary automatically leads to the appearance of an embedded trapezoid. Right panel: A microscopic defect in the boundary can prevent the GUE--corners asymptotics, because we would no longer have an interlacing array of horizontal lozenges.
\label{Figure_e_trap}}
\end{figure}

The conditions under which families of domains admit embedded trapezoids are quite subtle. In general, they cannot be read off of the macroscopic shape $\mathfrak{R}$; indeed, as indicated by the right panel of Figure \ref{Figure_e_trap}, a single microscopic defect along the boundary has the capacity to destroy this property entirely. Therefore, we make no attempt in this paper (beyond Theorem \ref{Theorem_GUE_simple}) to classify under what circumstances this property holds. However, let us mention that the situation for polygonal domains,  as in Example \ref{Example_domains} and Figure \ref{Figure_GUE_cases}, seems promising: one could hope to find embedded trapezoids near most of the straight segments of the boundaries. On the technical level, proving this fact requires a careful analysis of the frozen regions of tilings and showing that with high probability the region occupied by lozenges of one type has not just small amount, but actually no lozenges of other two types. We will not address this point here; it will follow as a consequence of the future work \cite{AH}. In particular, \Cref{Theorem_GUE_general} and results in \cite{AH} will together imply convergence to the GUE--corners process in the four situations indicated by pink crosses in Figure \ref{Figure_GUE_cases}.

\section{Concentration estimate}

\subsection{Statement of the concentration bound}
 An important step in our proof of Theorems \ref{Theorem_GUE_simple} and \ref{Theorem_GUE_general} is the following concentration estimate (possibly of independent interest) for height functions of random tilings on general simply-connected domains.
	
	\begin{prop}
		
		\label{Proposition_varianceh}
		
		Fix a simply-connected domain $\mathfrak{R} \subset \mathbb{R}^2$ with piecewise smooth boundary; a function $\mathfrak{h}: \partial \mathfrak{R} \rightarrow \mathbb{R}$ admitting an admissible extension to $\mathfrak{R}$; and a real number $\varepsilon > 0$. Let $R_1, R_2, \ldots $ denote a family of simply-connected, tileable domains of $\mathbb{T}$, and for each integer $N \ge 1$ let $h_N: \partial R_N \rightarrow \mathbb{Z}$ denote an associated boundary height function; assume that $(R_N; h_N)$ converges to $(\mathfrak{R}, \mathfrak{h})$ in the sense of Definition \ref{Def_domain_convergence}. For each integer $N \ge 1$, let $H_N: R_N \rightarrow \mathbb{Z}$ denote the height function of a uniformly random tiling from $\Omega (R_N)$. Then there exists a constant $C> 1$ (depending on all the above data) such that for any integer $N \ge C$ and any vertex $v \in R_N$, we have $\Var H_N (v) < \varepsilon N$.
		
	\end{prop}

	\begin{remark}		
	\cite[Theorem 21]{LSRT} implies that $\Var H_N (v) \le cN$ for a certain explicit constant $c >0$. Our results shows that this constant can, in fact, be chosen to be arbitrarily small as long as $N$ is large enough.
	\end{remark}


We deduce \Cref{Proposition_varianceh} from the following statement, which is proven in the rest of this section.
	
	\begin{prop}
		
		\label{varianceh2}
		
		In the setting of Proposition \ref{Proposition_varianceh}, let $H_N'$ denote the height function of another uniformly random tiling from $\Omega (R_N)$, independent of $H_N$. For $N > C$, we have $\Var \big[ H_N(v) - H_N' (v) \big] < \varepsilon N$ for any $v \in R_N$.
	
	\end{prop}
	
	\begin{proof}[Proof of \Cref{Proposition_varianceh} assuming \Cref{varianceh2}]
		
		Since $H_N$ and $H_N'$ are independent, we have that $2 \Var H_N(v) = \Var \big[ H_N(v) - H_N' (v) \big] < \varepsilon N$, where \Cref{varianceh2} was applied to deduce the inequality.		
	\end{proof}

	Throughout the remainder of this section, we adopt the setting of \Cref{varianceh2}; abbreviate $H:=H_N$ and $H':=H_N'$; and define $F: R_N \rightarrow \mathbb{Z}$ by $F(v) = H(v) - H'(v)$. We refer to the left side of \Cref{functionf} for an example.

\begin{figure}
\begin{center}
\begin{tikzpicture}[
>=stealth,
auto,
style={
scale = .5,
}
]
\draw[dotted, gray] (0, 0) -- (7, 3.5);
\draw[dotted, gray] (1, -.5) -- (8, 3);
\draw[dotted, gray] (2, -1) -- (8, 2);
\draw[dotted, gray] (3, -1.5) -- (8, 1);
\draw[dotted, gray] (4, -2) -- (8, 0);
\draw[dotted, gray] (0, 1) -- (6, 4);
\draw[dotted, gray] (0, 2) -- (5, 4.5);
\draw[dotted, gray] (0, 3) -- (4, 5);

\draw[dotted, gray] (1, -.5) -- (1, 4.5);
\draw[dotted, gray] (2, -1) -- (2, 5);
\draw[dotted, gray] (3, -1.5) -- (3, 5.5);
\draw[dotted, gray] (4, -2) -- (4, 5);
\draw[dotted, gray] (5, -2.5) -- (5, 4.5);
\draw[dotted, gray] (6, -2) -- (6, 4);
\draw[dotted, gray] (7, -1.5) -- (7, 3.5);

\draw[dotted, gray] (0, 1) -- (6, -2);
\draw[dotted, gray] (0, 2) -- (7, -1.5);
\draw[dotted, gray] (0, 3) -- (8, -1);
\draw[dotted, gray] (0, 4) -- (8, 0);
\draw[dotted, gray] (1, 4.5) -- (8, 1);
\draw[dotted, gray] (2, 5) -- (8, 2);

\draw[-, thick, black] (0, 0) -- (5, -2.5);
\draw[-, thick, black] (0, 0) -- (0, 4);
\draw[-, thick, black] (5, -2.5) -- (8, -1);
\draw[-, thick, black] (0, 4) -- (3, 5.5);
\draw[-, thick, black] (8, -1) -- (8, 3);
\draw[-, thick, black] (8, 3) -- (3, 5.5);

\fill[fill=black] (0, 0) circle [radius = .1] node[left, scale = .5]{$0$};
\fill[fill=black] (1, -.5) circle [radius = .1] node[below, scale = .5]{$0$};
\fill[fill=black] (2, -1) circle [radius = .1] node[below, scale = .5]{$0$};
\fill[fill=black] (3, -1.5) circle [radius = .1] node[below, scale = .5]{$0$};
\fill[fill=black] (4, -2) circle [radius = .1] node[below, scale = .5]{$0$};
\fill[fill=black] (5, -2.5) circle [radius = .1] node[below, scale = .5]{$0$};

\fill[fill=black] (0, 1) circle [radius = .1] node[left, scale = .5]{$0$};
\fill[fill=black] (1, .5) circle [radius = .1] node[below = 1, scale = .5]{$0$};
\fill[fill=black] (2, 0) circle [radius = .1] node[below = 1, scale = .5]{$0$};
\fill[fill=blue] (3, -.5) circle [radius = .1] node[below = 1, scale = .5]{$1$};
\fill[fill=black] (4, -1) circle [radius = .1] node[below = 1, scale = .5]{$0$};
\fill[fill=black] (5, -1.5) circle [radius = .1] node[below = 1, scale = .5]{$0$};
\fill[fill=black] (6, -2) circle [radius = .1] node[below, scale = .5]{$0$};

\fill[fill=black] (0, 2) circle [radius = .1] node[left, scale = .5]{$0$};
\fill[fill=black] (1, 1.5) circle [radius = .1] node[below = 1, scale = .5]{$0$};
\fill[fill=black] (2, 1) circle [radius = .1] node[below = 1, scale = .5]{$0$};
\fill[fill=black] (3, .5) circle [radius = .1] node[below = 1, scale = .5]{$0$};
\fill[fill=black] (4, 0) circle [radius = .1] node[below = 1, scale = .5]{$0$};
\fill[fill=black] (5, -.5) circle [radius = .1] node[below = 1, scale = .5]{$0$};
\fill[fill=black] (6, -1) circle [radius = .1] node[below = 1, scale = .5]{$0$};
\fill[fill=black] (7, -1.5) circle [radius = .1] node[below, scale = .5]{$0$};

\fill[fill=black] (0, 3) circle [radius = .1] node[left, scale = .5]{$0$};
\fill[fill=brown!80!black] (1, 2.5) circle [radius = .1] node[below = 1, scale = .5]{$1$};
\fill[fill=black] (2, 2) circle [radius = .1] node[below = 1, scale = .5]{$0$};
\fill[fill=red] (3, 1.5) circle [radius = .1] node[below = 1, scale = .5]{$-1$};
\fill[fill=red] (4, 1) circle [radius = .1] node[below = 1, scale = .5]{$-1$};
\fill[fill=black] (5, .5) circle [radius = .1] node[below = 1, scale = .5]{$0$};
\fill[fill=black] (6, 0) circle [radius = .1] node[below = 1, scale = .5]{$0$};
\fill[fill=orange] (7, -.5) circle [radius = .1] node[below = 1, scale = .5]{$-1$};
\fill[fill=black] (8, -1) circle [radius = .1] node[right, scale = .5]{$0$};

\fill[fill=black] (0, 4) circle [radius = .1] node[left, scale = .5]{$0$};
\fill[fill=black] (1, 3.5) circle [radius = .1] node[above = 1, scale = .5]{$0$};
\fill[fill=black] (2, 3) circle [radius = .1] node[above = 1, scale = .5]{$0$};
\fill[fill=red] (3, 2.5) circle [radius = .1] node[above = 1, scale = .5]{$-1$};
\fill[fill=green] (4, 2) circle [radius = .1] node[above = 1, scale = .5]{$-2$};
\fill[fill=red] (5, 1.5) circle [radius = .1] node[above = 1, scale = .5]{$-1$};
\fill[fill=black] (6, 1) circle [radius = .1] node[above = 1, scale = .5]{$0$};
\fill[fill=orange] (7, .5) circle [radius = .1] node[above = 1, scale = .5]{$-1$};
\fill[fill=black] (8, 0) circle [radius = .1] node[right, scale = .5]{$0$};

\fill[fill=black] (1, 4.5) circle [radius = .1] node[above = 1, scale = .5]{$0$};
\fill[fill=black] (2, 4) circle [radius = .1] node[above = 1, scale = .5]{$0$};
\fill[fill=black] (3, 3.5) circle [radius = .1] node[above = 1, scale = .5]{$0$};
\fill[fill=red] (4, 3) circle [radius = .1] node[above = 1, scale = .5]{$-1$};
\fill[fill=red] (5, 2.5) circle [radius = .1] node[above = 1, scale = .5]{$-1$};
\fill[fill=black] (6, 2) circle [radius = .1] node[above = 1, scale = .5]{$0$};
\fill[fill=black] (7, 1.5) circle [radius = .1] node[above = 1, scale = .5]{$0$};
\fill[fill=black] (8, 1) circle [radius = .1] node[right, scale = .5]{$0$};

\fill[fill=black] (2, 5) circle [radius = .1] node[above = 1, scale = .5]{$0$};
\fill[fill=black] (3, 4.5) circle [radius = .1] node[above = 1, scale = .5]{$0$};
\fill[fill=black] (4, 4) circle [radius = .1] node[above = 1, scale = .5]{$0$};
\fill[fill=black] (5, 3.5) circle [radius = .1] node[above = 1, scale = .5]{$0$};
\fill[fill=black] (6, 3) circle [radius = .1] node[above = 1, scale = .5]{$0$};
\fill[fill=yellow!75!orange] (7, 2.5) circle [radius = .1] node[above = 1, scale = .5]{$1$};
\fill[fill=black] (8, 2) circle [radius = .1] node[right, scale = .5]{$0$};

\fill[fill=black] (3, 5.5) circle [radius = .1] node[above = 1, scale = .5]{$0$};
\fill[fill=black] (4, 5) circle [radius = .1] node[above = 1, scale = .5]{$0$};
\fill[fill=black] (5, 4.5) circle [radius = .1] node[above = 1, scale = .5]{$0$};
\fill[fill=black] (6, 4) circle [radius = .1] node[above = 1, scale = .5]{$0$};
\fill[fill=black] (7, 3.5) circle [radius = .1] node[above = 1, scale = .5]{$0$};
\fill[fill=black] (8, 3) circle [radius = .1] node[right, scale = .5]{$0$};

\end{tikzpicture}
\qquad\qquad
\begin{tikzpicture}[
>=stealth,
auto,
style={
scale = .5
}
]

\draw[thick, ->,gray] (18, 6.5) -- (18, 3.8);
\draw[thick, ->,gray] (18, 3.5) -- (18, .8);
\draw[thick, ->,gray] (18, 6.5) -- (13, 3.8);
\draw[thick, ->,gray] (18, 6.5) -- (15.5, 3.8);
\draw[thick, ->,gray] (18, 6.5) -- (20.5, 3.8);
\draw[thick, ->,gray] (18, 6.5) -- (23, 3.8);

\fill[fill=black] (18, 6.5) circle[radius = .3] node[above = 5, scale = .5]{$F(S_1) = 0$};
\fill[fill=blue] (13, 3.5) circle[radius = .3] node[below = 5, scale = .5]{$F(S_2) = 1$};
\fill[fill=brown!80!black] (15.5, 3.5) circle[radius = .3] node[below = 5, scale = .5]{$F(S_3) = 1$};
\fill[fill=red] (18, 3.5) circle[radius = .3] node[below = 5, scale = .5]{$F(S_4) = -1$};
\fill[fill=green] (18, .5) circle[radius = .3] node[below = 5, scale = .5]{$F(S_5) = -2$};
\fill[fill=orange] (20.5, 3.5) circle[radius = .3] node[below = 5, scale = .5]{$F(S_6) = -1$};
\fill[fill=yellow!75!orange] (23, 3.5) circle[radius = .3] node[below = 5, scale = .5]{$F(S_7) = 1$};
\end{tikzpicture}
\end{center}

\caption{\label{functionf} Left panel: an example for the difference function $F$ on a hexagonal domain (the original associated height functions $H$ and $H'$ are not shown). Right panel: the associated distance graph.}
\end{figure}
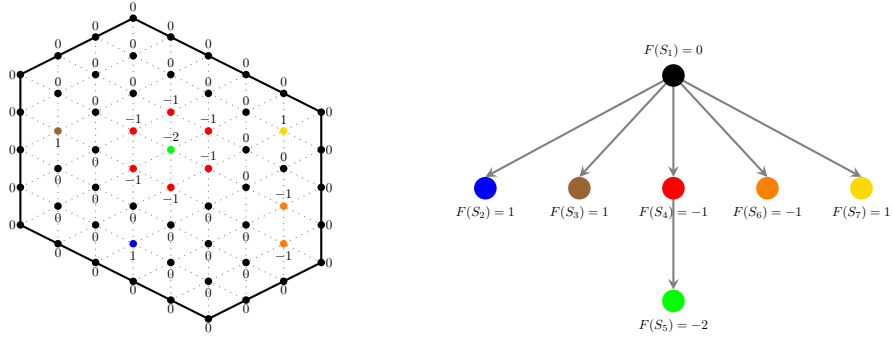

	\subsection{Distance Graphs}
	
	\label{HeightConcentrate}

	A family $\mathscr{S} = (S_1, S_2, \ldots , S_m)$ of mutually disjoint, connected subsets with $\bigcup_{i = 1}^m S_i = R_N$ is called a \emph{level set decomposition of $R_N$ with respect to $F$} if the following hold.
	
	\begin{enumerate}
		
		\item If $v, v' \in S_i$ for some $1 \le i \le m$, then $F(v) = F(v')$.
		\item If $v \in S_i$ and $v' \in S_j$ for some $1 \le i < j \le m$, and if $v$ and $v'$ are adjacent, then $F(v) \ne F(v')$.
	
	\end{enumerate}

	\noindent We then set $F (S_i) = F(v)$ for any $v \in S_i$. We further say that $S_i, S_j \in \mathscr{S}$ are \emph{adjacent} if $i \ne j$ and there exist $v \in S_i$ and $v' \in S_j$ such that $v$ and $v'$ are adjacent in $\mathbb{T}$. Observe, in particular, that
	 \begin{flalign}
	 	\label{ijgsgs}
	\big| F(S_i) - F(S_j) \big| = 1, \qquad \text{whenever $S_i$ and $S_j$ are adjacent}.
	\end{flalign}

	\noindent We additionally say that $S_j$ is \emph{exterior adjacent} to $S_i$ if $S_j$ is contained in the infinite connected component of $\mathbb{T} \setminus S_i$; we otherwise call $S_j$ \emph{interior adjacent} to $S_i$.
	
	\begin{lem}
		
	\label{sjexterior}
	
	For any $i \in [1, m]$, there is at most one element of $\mathscr{S}$ that is exterior adjacent to $S_i$.
	
	 \end{lem}

 	\begin{proof}
 		
 	Let $\mathcal{C}$ denote the (unique) infinite connected component of $\mathbb{T} \setminus S_i$; set $\partial \mathcal{C} = \{ v_1, v_2, \ldots , v_K \}$. Since $S_i$ is connected, we may order the $\{ v_k \}$ such that $v_k$ is adjacent to $v_{k + 1}$ for each $k \in [1, K]$, where $v_{K + 1} = v_1$. Then, since each $v_k$ is not in $S_i$ but is adjacent to a vertex in $S_i$, \eqref{ijgsgs} implies that $\big| F(v_k) - F(S_i) \big| = 1$ for each $k \in [1, K]$. In particular, all of the $F(v_k)$ are of the same parity, and so $\big| F(v_k) - F(v_{k + 1}) \big| \ne 1$ for each $k$. Since $\{ v_k, v_{k + 1} \}$ are adjacent it follows again from \eqref{ijgsgs} that $v_k$ and $v_{k + 1}$ reside in the same $S_j$. Hence, each $v_k \in \mathcal{C}$ belongs to the same $S_j$, meaning that there is at most one element of $\mathscr{S}$ that is exterior adjacent to $S_i$.
 	\end{proof}

	Given this notation and \Cref{sjexterior}, we define the distance graph.
	
	\begin{definition}
		
		\label{graphg}
		
	Define the (random) \emph{distance graph} $\mathcal{G} = \mathcal{G}_{\mathscr{S}}$, which is a rooted, directed tree with $m$ vertices that are indexed by the elements of $\mathscr{S}$, such that $S_i$ is a child of $S_j$ if and only if $S_j$ is exterior adjacent to $S_i$. The root of the tree is the unique element of $\mathscr{S}$ that contains the boundary $\partial R_N$ of $R_N$.
	
	\end{definition}

	Observe that $\mathcal{G}_{\mathscr{S}}$ is measurable with respect to $\mathscr{S}$. An example is depicted on the right side of \Cref{functionf}. In what follows, we let $\dist_{\mathcal{G}} (S_i, S_j)$ denote the distance between $S_i$ and $S_j$ under $\mathcal{G}$, and let $\dist_{\mathcal{G}} (u, v) = \dist_{\mathcal{G}} (S_i, S_j)$ whenever $v \in S_i$ and $u \in S_j$.
	
	The use of these notions is provided by the following lemma.

	\begin{lem}
		
		\label{uvf}
		
		For any $u, v \in R_N$, we have $\Var \big[ F(u) - F(v) \big| \mathscr{S} \big] = \dist_{\mathcal{G}} (u, v)$.
	
	\end{lem}

	\begin{proof}
		
	Conditional on $\mathscr{S}$, the law of $F$ is uniform on the set of all integer-valued functions on $\mathscr{S}$ satisfying the two conditions:
 \begin{enumerate}
   \item The value of $F$ on the root of $\mathcal{G}_{\mathscr{S}}$ is $0$. (Because $F |_{\partial R_N} = H |_{\partial R_N} - H' |_{\partial R_N} = 0$.)
   \item $\big| F(S_i) - F(S_j) \big| = 1$ holds whenever $S_i$ is adjacent to $S_j$ (by \eqref{ijgsgs}).
  \end{enumerate}
 This description follows from the uniformity of the measure on lozenge tilings that we started from.
  Because $\mathcal{G}_{\mathscr{S}}$ is a tree, we can sample $F$ as follows: Associate with each edge $e$ of $\mathcal{G}$ a Bernoulli random variable $B_e$ such that $\mathbb{P} [B_e = 1] = \frac{1}{2} = \mathbb{E}[B_e = -1]$. Then, letting $S_i \in \mathscr{S}$ denote the level set containing $v$, set $F(v) = F(S_i)$ equal to $\sum_{e \in \mathcal{P}} B_e$, where $\mathcal{P}$ is the unique path of edges from the root of $\mathcal{G}$ to $S_i$. Taking variances in this equality, we deduce the lemma.
	\end{proof}

	In view of \Cref{uvf}, to show \Cref{varianceh2} it suffices to bound $\mathbb{E} \big[ \dist_{\mathcal{G}} (u, v) \big] = o (N)$; let us mention that the bound $\mathbb{E} \big[ \dist_{\mathcal{G}} (u, v) \big] = O(N)$ follows from the (deterministic) fact that $\dist_{\mathcal{G}} (u, v) \le |u - v|$.

	\subsection{Preliminary Results }
	
	\label{Measures}

	In this section we recall several known results from \cite{VPT,MCFRS,CTFMD,ULS} that will assist in the proof of the improved $o(N)$ estimate on $\mathbb{E} \big[ \dist_{\mathcal{G}} (u, v) \big]$. We start from the results concerning the limiting height function and local statistics of random tilings. To that end, for any $x \in \mathbb{R}_{\ge 0}$ and $(s, t) \in \overline{\mathcal{T}}$, set
	\begin{flalign*}
		L(x) = - \displaystyle\int_0^x \log |2 \sin z| dz; \qquad \sigma (s, t) = \displaystyle\frac{1}{\pi} \Big( L(\pi s) + L (\pi t) + L \big( \pi (1 - s - t) \big) \Big).
	\end{flalign*}
	
	\noindent For any $\mathfrak{F} \in \Adm (\mathfrak{R})$, define
	\begin{flalign*}
		\mathcal{E} (\mathfrak{F}) = \displaystyle\int_{\mathfrak{R}} \sigma \big( \nabla \mathfrak{F} (z) \big) dz.
	\end{flalign*}
	
	Letting $\mathfrak{H} \in \Adm (\mathfrak{R}; \mathfrak{h})$ denote the maximizer of $\mathcal{E}$ on $\Adm(\mathfrak{R})$ with boundary data $\mathfrak{h}$, the following result of \cite{VPT} claims that $H$ converges to $\mathfrak{H}$ with high probability.
	
	\begin{lem}[{\cite[Theorem 1.1]{VPT}}]
	
	\label{Lemma_var_principle}
	
		In the setting of Proposition \ref{Proposition_varianceh}, for any $\varepsilon > 0$, there exists a constant $C>1$ such that if $N > C$ then
		\begin{flalign*}
			\mathbb{P} \Big[ \displaystyle\max_{v \in R_N\cap N\mathfrak{R}} \big| \tfrac{1}{N} H_N (v) - \mathfrak{H} (\tfrac{1}{N} v) \big| > \varepsilon \Big] < \varepsilon.
		\end{flalign*}
	\end{lem}

	We will later use the following lemma from \cite{MCFRS} that provides a deterministic property for the limit shape $\mathfrak{H}$ concerning the discontinuities of its gradient. It essentially states that at any discontinuity of $\nabla \mathfrak{H}$ there is a path to $\partial \mathfrak{R}$, parallel to an axis of $\mathbb{T}$, along which $\nabla \mathfrak{H}$ is mimimal or maximal.
	
	\begin{lem}[{\cite[Theorem 1.3]{MCFRS}}]
		
		\label{nablaf}
		
		For any $z \in \mathfrak{R}$, we either have that $\nabla \mathfrak{H}$ is continuous at $z$ or there exists some $z' \in \partial \mathfrak{R}$, with $z' - z = tw$ for some $t \in \mathbb{R}_{> 0}$ and $w \in \big\{ (1, 0), (0, 1), (1, 1) \big\}$, such that the following holds. If $w \in \big\{ (1, 0), (0, 1) \big\}$, then $\mathfrak{H} (z + sw) = \mathfrak{H} (z)$ for each $s \in [0, t]$. If $w = (1, 1)$, then $\mathfrak{H} (z + sw) = \mathfrak{H} (z) + s$ for each $s \in [0, t]$.
		
	\end{lem}

	Next, although in this paper we are primarily interested in uniformly random lozenge tilings of finite domains, it will be useful to state several properties about certain infinite-volume measures on tilings of $\mathbb{T}$ that will arise as limit points of local statistics for our model. Therefore, let us briefly recall a two-parameter family of such measures from \cite{LSLD,okounkov_reshetikhin_3d,R} (see also \cite[Section 5]{OD} and \cite[Lecture 13]{Gorin_book}). To that end, let $\mathfrak{P} (R)$ denote the space of probability measures on the set $\Omega (R)$ of all lozenge tilings of some domain $R \subseteq \mathbb{T}$. We say that $\mu \in \mathfrak{P} (\mathbb{T})$ is a \emph{Gibbs measure} if it satisfies the following property for any finite domain $R \subset \mathbb{T}$. The probability under $\mu$ of selecting any $\mathscr{M} \in \Omega (\mathbb{T})$, conditional on the restriction of $\mathscr{M}$ to $\mathbb{T} \setminus R$, is uniform.
	
	Moreover, for any $w \in \mathbb{Z}^2$, define the translation map $\mathfrak{S}_w: \mathbb{Z}^2 \rightarrow \mathbb{Z}^2$ by setting $\mathfrak{S}_w (v) = v - w$ for any $v \in \mathbb{Z}^2$. Then $\mathfrak{S}_w$ induces an operator on $\mathfrak{P} (\mathbb{T})$ that translates a tiling by $-w$; we also refer it to by $\mathfrak{S}_w$. A measure $\mu \in \mathfrak{P} (\mathbb{T})$ is called \emph{translation-invariant} if $\mathfrak{S}_w \mu = \mu$, for any $w \in \mathbb{Z}^2$.  We further call a Gibbs translation-invariant measure $\mu \in \mathfrak{P} (R)$ \emph{ergodic} if, for any $p \in (0, 1)$ and Gibbs translation-invariant measures $\mu_1, \mu_2 \in \mathfrak{P} (R)$ such that $\mu = p \mu_1 + (1 - p) \mu_2$, we have $\mu_1 = \mu = \mu_2$.
	
	Let $\mu \in \mathfrak{P} (\mathbb{T})$ be a translation-invariant measure, and let $H_{\mu}$ denote the height function associated with a randomly sampled lozenge tiling under $\mu$, such that $H_{\mu} (0, 0) = 0$. Setting $s = \mathbb{E} \big[ H_{\mu} (1, 0) - H_{\mu} (0, 0) \big]$ and $t = \mathbb{E} \big[ H_{\mu} (0, 1) - H_{\mu} (0, 0) \big]$, the \emph{slope} of $\mu$ is defined to be the pair $(s, t)$. We must then have $(s, t) \in \overline{\mathcal{T}}$, since $s$, $t$, and $1 - s - t$ denote the probabilities of a given lozenge in a random tiling (under $\mu$) being of types $1$, $2$, and $3$, respectively.
	
	It was shown in \cite[Theorem 9.1.1]{R} that, for any $(s, t) \in \overline{\mathcal{T}}$, there exists a unique ergodic Gibbs translation-invariant measure $\mu = \mu_{s, t} \in \mathfrak{P} (\mathbb{T})$ of slope $(s, t)$. We then have the following two results concerning this measure; the first one bounds the variance of the height function under any such $\mu_{s, t}$, and the second one shows convergence of local statistics of $H$ to $\mu_{s, t}$.
	
	In the below, we let $\mathscr{M}$ denote the law of the tiling of $R_N$ associated with $H$ and, for any $v \in \mathbb{T}$ and $D \in \mathbb{Z}_{\ge 0}$, we let $\mathcal{B}_D (v) = \big\{ u \in \mathbb{T} : |u - v| \le D \}$ denote the disk centered at $v$ with radius $D$. By $\mathscr{M} |_{\mathcal{B}_D (v)}$ we mean the law of the tiling inside the disk $\mathcal{B}_D (v)$.
Moreover, for any $\varepsilon > 0$, we let $\mathcal{T}_{\varepsilon} = \big\{ (s, t) \in \mathbb{R}_{> 0}^2 : \varepsilon < s + t < 1 - \varepsilon \big\}$, which describes the set of slopes that are ``uniformly liquid.'' Additionally, for any measures $\mu$ and $\nu$ on a probability space $\Omega$ with $\sigma$--algebra $\mathcal A$, we let
$$
		d_{\TV} (\mu, \nu) = \displaystyle\max_{A \in \mathcal A} \big| \mu (A) - \nu (A) \big|
$$
 denote the total variation distance between $\mu$ and $\nu$.

	\begin{lem}[{\cite[Proposition 4.1]{CTFMD}}]
		
		\label{hmoment}
		
		There exists a constant $C > 1$ such that the following holds. For any $(s, t) \in \mathcal{T}$ and $u, v \in \mathbb{T}$ we have under $\mu = \mu_{s, t}$ that $\Var \big( H_{\mu} (u) - H_{\mu} (v) \big) \le C \log |u - v|$.
		
	\end{lem}

	Let us mention that a logarithmic variance bound as in \Cref{hmoment} was also established as Theorem 4.5 of \cite{DA}, but there the uniformity in the slope $(s, t)$ was not mentioned. \cite{CTFMD} did not mention lozenge tilings explicitly, but the discrete sine process studied there is another face of the same object, see \cite[Section 3.1.1]{okounkov_reshetikhin_3d} or \cite[Section 13.4]{Gorin_book}.

	\begin{lem}[{\cite{ULS}}]
		
		\label{convergemu}
		
		In the setting of Proposition \ref{Proposition_varianceh}, for any real numbers $\varepsilon, \omega > 0$ and $D > 1$, there exists a constant $C> 1$ such that the following holds for any $N > C$. Let $v \in R_N$ be a vertex such that $\mathfrak{v} = \tfrac{1}{N} v \in \mathfrak{R}$ satisfies $\dist (\mathfrak{v}, \partial \mathfrak{R}) \ge \varepsilon$ and $(s, t) = \nabla \mathfrak{H} (\mathfrak{v}) \in \mathcal{T}_{\varepsilon}$. Then, $d_{\TV} \big( \mathscr{M} |_{\mathcal{B}_D (v)}, \mu_{s, t} |_{\mathcal{B}_D (v)} \big) < \omega$.
		
	\end{lem}
	
	\begin{remark}
		
		\label{convergev}
		
		The uniformity in $v$ (and $\mathfrak{v}$) of the constant $C$ stated in \Cref{convergemu} was not directly claimed in (Theorem 1.5 of) \cite{ULS} but follows from its proof. We outline this in \Cref{ProofConverge} below.
		
	\end{remark}

	\subsection{Proof of \Cref{varianceh2}}
	
	\label{EstimateProof}
	
	In this section we establish \Cref{varianceh2}. We begin with the following lemma that bounds $\dist_{\mathcal{G}} (u, v) = o \big( |u - v| \big)$ when $u - v$ are nearby.

	\begin{lem}
		
	\label{distancegestimate}

	There exists a constant $C_1>1$ such that for any real numbers $\varepsilon \in (0, 1)$ and $D > 2$, there is another constant $C_2>1$ (depending on all the data) and the following holds for any $N \ge C_2$: Fix $u, v \in R_N$ with $|u - v| \le D$ satisfying $\dist (\tfrac{1}{N} v, \partial \mathfrak{R}) \ge \varepsilon$ and $\nabla \mathfrak{H} (\tfrac{1}{N} u) \in \mathcal{T}_{\varepsilon}$. Then
	\begin{flalign*}
	\mathbb{P} \big[ \dist_{\mathcal{G}} (u, v) > C_1 (\log D)^2 \big] \le (\log D)^{-1}.
	\end{flalign*}
	
	\end{lem}

	\begin{proof}
		
		Let $\omega =D^{-1}$ and set $(s, t) = \nabla \mathcal{H} (\tfrac{1}{N} u)$. Letting $\mathscr{M}$ and $\mathscr{M}'$ denote the tilings of $R_N$ associated with $H$ and $H'$, respectively, \Cref{convergemu} implies that $\mathscr{M}$ and $\mathscr{M}'$ can be coupled with $\mu_{s, t}$ such that they all coincide on $\mathcal{B}_D (v)$ with probability at least $1 - \omega$ for sufficiently large $N$. Hence, using the fact that $H$ is $1$-Lipschitz, \Cref{hmoment} implies the existence of a constant $C > 1$ such that, for any fixed $D > 2$ and sufficiently large $N$, we have
		\begin{flalign*}
			\Var \big[ H(u) - H(v) \big] < C \log D + 2\omega D, \quad \text{and} \quad \Var \big[ H' (u) - H' (v) \big] < C \log  D + 2 \omega D,
		\end{flalign*}
		
		\noindent whenever $|u - v| \le D$. Adding these bounds, using the independence of $H$ and $H'$, and $w=D^{-1}$, we obtain
		\begin{flalign*}
			\mathbb{E} \big[ \dist_{\mathcal{G}} (u, v) \big] = \Var \big[ F(u) - F(v) \big] < 2C \log D + 2,
		\end{flalign*}
		
		\noindent where to deduce the first equality we used \Cref{uvf}. It remains to use the Markov's inequality.
	\end{proof}
	
	The next step is to deduce a lemma bounding $\dist_{\mathcal{G}} (u, v) = o(N)$ on global scales in the liquid region; it will follow from applying \Cref{distancegestimate} on all bounded intervals on a path between $u$ and $v$. 	
		
	\begin{lem}
		
		\label{distancegestimate2}
		
		For any $\varepsilon \in (0, 1)$, there exist constants $C = C (\varepsilon) > 1$ and $C'> 1$ (depending on all the data) such that the following holds for any $N \ge C'$: Fix $u, v \in R_N$ with $|u - v| > C$, such that $\dist (\tfrac{1}{N} u, \partial \mathfrak{R}) \ge \varepsilon$ and $\dist (\tfrac{1}{N} v, \partial \mathfrak{R}) \ge \varepsilon$. Further suppose that there exists an open disk $\mathfrak{U} \subset \mathfrak{R}$ containing $\tfrac{1}{N} u$ and $\tfrac{1}{N} v$, with $\nabla \mathfrak{H} (z) \in \mathcal{T}_{\varepsilon}$ for each $z \in \mathfrak{U}$. Then, $\mathbb{P} \big[ \dist_{\mathcal{G}} (u, v) > \varepsilon |u - v| \big] \le \varepsilon$.
		
	\end{lem}
	
	\begin{proof}
		
		Let $D > 2$ denote some integer, which will only depend on $\varepsilon$ and $\mathfrak{R}$, and will be fixed later. Then, there exists a sequence $u = v_0, v_1, \ldots , v_k = v \in N \mathfrak{U} \cap R_N$ such that $|v_i - v_{i - 1}| \le D$ for each $i \in [1, k]$ and $k \le \big\lceil 2 D^{-1} |u - v| \big\rceil$. Recall the constant $C_1 > 1$ from \Cref{distancegestimate}, and define for each $i \in [1, k]$ the event
		\begin{flalign*}
		\mathscr{E}_i = \big\{ \dist_{\mathcal{G}} (v_{i - 1}, v_i) > C_1 (\log D)^2 \big\}.
		\end{flalign*}
	
		Then, \Cref{distancegestimate} implies $\mathbb{P} \big[ \mathscr{E}_i \big] \le (\log D)^{-1}$. So, since denoting by $E^c$ the complement of any event $E$ we have
		\begin{flalign*}
			\dist_{\mathcal{G}} (u, v) \le \displaystyle\sum_{i = 1}^k \big( |v_i - v_{i + 1}| \textbf{1}_{\mathscr{E}_i} + C_1 (\log D)^2 \textbf{1}_{\mathscr{E}_i^c} \big),
		\end{flalign*}
		
		\noindent it follows that
		\begin{flalign*}
		\mathbb{E} \big[ \dist_{\mathcal{G}} (u, v) \big] & \le \displaystyle\sum_{i = 1}^k D \mathbb{P} [\mathscr{E}_i] + C_1 k (\log D)^2 \\
		& \le kD (\log D)^{-1} + C_1 k (\log D)^2 \\
		& \le D (\log D)^{-1} \big\lceil 2D^{-1} |u - v| \big\rceil + C_1 (\log D)^2 \big\lceil 2 D^{-1} |u - v| \big\rceil \\
		& \le 4 |u- v| \big( (\log D)^{-1} + C_1 D^{-1} (\log D)^2 \big),
		\end{flalign*}
	
		\noindent whenever $|u - v| \ge D$. Letting $C = C (\varepsilon) > 1$ so that $4 (\log C)^{-1} + 4 C_1 C^{-1} (\log C)^2 < \varepsilon^2$, we deduce by taking $D = C$ that $\mathbb{E} \big[ \dist_{\mathcal{G}} (u, v) \big] < \varepsilon^2 |u - v|$, and so the lemma follows from the Markov's inequality.
	\end{proof}

	The next lemma will be used in settings complimentary to Lemma \ref{distancegestimate2}. The claim is that if $u, v \in R_N$ are such that $v - u$ is parallel to an axis of $\mathbb{T}$ and $H$ has approximately maximal or minimal slope along $v - u$, then $\dist_{\mathcal{G}} (u, v)$ is small.

	\begin{lem}
		
		\label{distanceg2}
		
		Let $\delta \in (0, 1)$ be a real number, and let $u, v \in R_N$ denote vertices such that $v - u = tw$ for some $t \in \mathbb{Z}_{\ge 0}$ and $w \in \big\{ (1, 0), (0, 1), (1, 1) \big\}$. Define the events
		\begin{flalign}
			\label{f}
			\begin{aligned}
			\mathcal{F}_1 = \mathcal{F}_1 (u, v; \delta) & = \Big\{ H(v) - H(u) \le \delta t \Big\} \cap  \Big\{ H' (v) - H' (u) \le \delta t \Big\}; \\
			 \mathcal{F}_2 = \mathcal{F}_2 (u, v; \delta) & = \Big\{ H(v) - H(u) \ge (1 - \delta) t \Big\} \cap \Big\{ H' (v) - H' (u) \ge (1 - \delta) t \Big\}.
			 \end{aligned}
		 		\end{flalign}
	
		\noindent On $\mathcal{F}_1 \cup \mathcal{F}_2$, we have $\dist_{\mathcal{G}} (u, v) \le 2 \delta t$.
		
	\end{lem}
	
	\begin{proof}
		
		 In this proof, let us only assume that $\mathcal{F}_1$ holds, for the verification in the case when $\mathcal{F}_2$ holds is entirely analogous. Set $u_j = u + j w \in R_N$ for each $j \in [0, t]$, so that $u_t = v$.	 Note that both $H(u_j)$ and $H'(u_j)$ depend on $j$ in a monotone way by our definition of the height function and the choice of $w$. Hence, using \eqref{f}, we conclude that there exist at most $\delta t$ indices $i \in [1, t]$ such that $H (u_i) -  H (u_{i - 1}) \ne 0$ and at most $\delta t$ indices $i \in [1, t]$ such that $H' (u_i) - H' (u_{i - 1}) \ne 0$. Thus, there exist at most $2 \delta t$ indices $i \in [1, t]$ for which $F(u_i) = H (u_i) - H' (u_i) \ne H (u_{i - 1}) - H' (u_{i - 1}) = F(u_{i - 1})$. Hence, for all but $(1 - 2 \delta) t$ indices $i \in [1, t]$, we must have that $u_i$ and $u_{i - 1}$ are in the same connected component, which implies that at most $2 \delta t$ connected components separate $u$ from $v$.
	\end{proof}

	Using the above results, the following proposition bounds the (graph) distance between any vertex in $R_N$ and the root of $\mathcal{G}$.

	\begin{prop}
		
		\label{distanceguv}
	
		For any real number $\varepsilon > 0$, there exists a constant $C> 1$ such that for any $N > C$, we have that $\mathbb{P} \big[ \dist_{\mathcal{G}} (v, \partial R) > \varepsilon N \big] < \varepsilon$ for each $v \in R_N$.
		
	\end{prop}
	
	\begin{proof}
		
		Fix a real number $\delta > 0$ with $\delta < \tfrac{\eps}{25} (1+ \diam \mathfrak{R})^{-1} $, and let $\mathfrak{R}_{\delta} = \big\{ z \in \mathfrak{R}: d (z, \partial \mathfrak{R}) > \delta \big\}$. Let $\mathfrak{v} = \tfrac{1}{N} v$. We may assume that $\mathfrak{v} \in \mathfrak{R}_{\delta}$, for otherwise for large $N$ we have a deterministic bound
$$\dist_{\mathcal{G}} (v, \partial R_N) \le d (v, \partial R_N) \le 2 \delta N < \varepsilon N.$$

Our plan is to define a path $\{ \zeta_0, \zeta_1, \ldots , \zeta_k \} \subset \mathfrak{R}\cap \tfrac{1}{N}R_N$ from $\mathfrak{v} = \zeta_0$ to $\zeta_k \in \mathfrak{R} \setminus \mathfrak{R}_{\delta}$, so that\footnote{The points $N\zeta_i$ in  $\dist_{\mathcal{G}}(\cdot,\cdot)$ need to be vertices in $\mathbb{T}$, i.e., they should be confined to the integer lattice. Hence, we formally should  write $\dist_{\mathcal{G}} (\lfloor N \zeta_{i - 1}\rfloor , \lfloor N \zeta_i\rfloor)$. For notational convenicence we are going to omit the integer parts throughout this proof; this will not affect the validity of the argument.} $\dist_{\mathcal{G}} (N \zeta_{i - 1}, N \zeta_i) \le 2 \delta N |\zeta_i - \zeta_{i - 1}|$ with high probability, for each $i \in [1, k]$.
		To that end, for each $z \in \mathfrak{R}$, we define an open subset $\mathfrak{U} (z) \subset \mathfrak{R}$ as follows. If $\nabla \mathfrak{H}$ is continuous at $z$ then let $\mathfrak{U} (z)$ denote a disk, of radius at most $\frac{\delta}{2}$, centered at $z$ such that $\big| \nabla \mathfrak{H} (z') - \nabla \mathfrak{H} (z) \big| < \delta$ for each $z' \in \overline{\mathfrak{U} (z)}$. If $\nabla \mathfrak{H}$ is discontinuous at $z$, then \Cref{nablaf} implies that there exists $z' \in \partial \mathfrak{R}$ such that the conditions there hold; then let $\mathfrak{U} (z)$ denote the $\delta$-neighborhood of the segment connecting $z$ and $z'$.
		
Because $\bigcup_{z \in \mathfrak{R}_{\delta}} \mathfrak{U} (z)$ is an open cover of the compact domain $\overline{\mathfrak{R}}_{\delta}$, it admits a finite subcover $\bigcup_{i = 1}^K \mathfrak{U} (z_i)$. We call $z = z_i$ of type 1 if $\nabla \mathfrak{H}$ is discontinuous at $z_i$; of type 2 if $\nabla \mathfrak{H} (z) \in \mathcal{T}_{2 \delta}$ and $\nabla \mathfrak{H}$ is continuous at $z$; and of type $3$ if $\nabla \mathfrak{H} (z) \notin \mathcal{T}_{2 \delta}$ and $\nabla \mathfrak{H}$ is continuous at $z$. If $z$ is of type 3, then $\nabla \mathfrak{H} (z)$ is of distance at most $2 \delta$ from either the segment $\big[ (0, 0), (1, 0) \big]$, $\big[ (0, 0), (0, 1)\big]$, or $\big[ (1, 0), (0, 1) \big]$ on $\partial \mathcal{T}$; we call $z$ of types 3(a), 3(b), and 3(c) in these cases, respectively.

		Now we define the sequence $\mathfrak{v} = \zeta_0, \zeta_1, \zeta_2, \ldots , \zeta_k  \in \mathfrak{R}\cap \tfrac{1}{N}R_N$ inductively as follows. First set $\zeta_0 = \mathfrak{v}$, and suppose $\zeta_i$ is defined for some $i \ge 0$. If $\zeta_i \in \mathfrak{R}\setminus \mathfrak{R}_{\delta}$, then set $k = i$ and stop. If otherwise $\zeta_i \in \mathfrak{R}_{\delta}$, then there exists some $z_j = z_{j (i)}$ such that $\zeta_i \in \mathfrak{U} (z_j)$. We will define $\zeta_{i + 1}$ depending on the type of $z_j$.
		
		\begin{enumerate}
			\item If $z_j$ is of type $1$, then let $z' \in \partial \mathfrak{R}$ and $w \in \big\{ (1, 0), (0, 1), (1, 1) \big\}$ be the point and vector corresponding to $z = z_j$ under \Cref{nablaf}, respectively. Set $k = i + 1$ and $\zeta = \zeta_k = z' - \frac{\delta}{2} w$.
			\item If $z_j$ is of type $2$, then let $\zeta_{i + 1} \in \partial \mathfrak{U} (z_j)$ denote the point such that $\zeta_{i + 1} - \zeta_i$ is a positive multiple of $(1, 1)$.
			\item Suppose $z_j$ is of type 3.
			\begin{enumerate}
				\item If $z_j$ is of type 3(a), then let $\zeta_{i + 1} \in \partial \mathfrak{U} (z_j)$ denote the point such that $\zeta_{i + 1} - \zeta_i$ is a positive multiple of $(0, 1)$.
				\item If $z_j$ is of type 3(b), then let $\zeta_{i + 1} \in \partial \mathfrak{U} (z_j)$ denote the point such that $\zeta_{i + 1} - \zeta_i$ is a positive multiple of $(1, 0)$.
				\item If $z_j$ is of type 3(c), then let $\zeta_{i + 1} \in \partial \mathfrak{U} (z_j)$ denote the point such that $\zeta_{i + 1} - \zeta_i$ is a positive multiple of $(1, 1)$.
			\end{enumerate}
		\end{enumerate}
	
		\noindent Let us make three comments about this definition. First, observe when $z_j$ is of type 3 that $\zeta_{i + 1}$ is chosen so that $\zeta_{i + 1} - \zeta_i$ is approximately orthogonal to $\nabla \mathfrak{H} (z_i)$, namely, $\mathfrak{H} (\zeta_i) \approx \mathfrak{H} (\zeta_{i + 1})$. Second, we have $\zeta_{i + 1} \notin \mathfrak{U} \big( z_{j(i)} \big)$ (since $\mathfrak{U} \big( z_{j(i)} \big)$ is open), so $\zeta_{i + 1} \ne \zeta_i$. In particular, since $\zeta_{i + 1} - \zeta_i \in \mathbb{R}_{\ge 0}^2$, this implies that the process described above to define the sequence $\{ \zeta_0, \zeta_1, \ldots , \zeta_k \}$ eventually ends. Third, as can be verified by induction on $j$, we have $\dist (\zeta_j, \partial \mathfrak{R}_{\delta}) \ge \frac{\delta}{2}$ for each $j \in [0, k]$, so that $N \zeta_k \in R_N = R$ for sufficiently large $N$.
		
		Next, fix some $i \in [0, k - 1]$, and set $\widetilde{u} = N \zeta_i$ and $\widetilde{v} = N \zeta_{i + 1}$; further denote $z_j = z_{j (i)}$. We claim that
		\begin{flalign}
			\label{pguv}
			\begin{aligned}
			& \mathbb{P} \big[ \dist_{\mathcal{G}} (\widetilde{u}, \widetilde{v}) > 4 \delta N (\diam \mathfrak{R} + 1) \big] < \delta k^{-1}, \qquad \text{if $z_j$ is of type $1$}; \\
			& \mathbb{P} \big[ \dist_{\mathcal{G}} (\widetilde{u}, \widetilde{v}) > 8 \delta |\widetilde{u} - \widetilde{v}| \big] < \delta k^{-1}, \qquad \qquad \qquad \text{if $z_j$ is of type $2$ or $3$}.
			\end{aligned}
		\end{flalign}
	
		Before proving \eqref{pguv}, let us show how it implies the statement of Proposition \ref{distanceguv}. To that end, since $k = i + 1$ whenever $z_{j(i)}$ is of type 1, there is at most one index $i$ such that $z_j = z_{j(i)}$ is of type 1. Hence, applying \eqref{pguv} and a union bound, we deduce
		\begin{flalign}
			\label{vprobabilityr}
			\begin{aligned}
			\mathbb{P} \big[ \dist_{\mathcal{G}} (v, \partial R_N) > \varepsilon N \big] & \le \mathbb{P} \big[ \dist_{\mathcal{G}} (v, N \zeta_k) > (\varepsilon - 2 \delta) N \big] \\
			& \le \displaystyle\sum_{i = 1}^{k - 1} \mathbb{P} \big[ \dist_{\mathcal{G}} (N \zeta_{i - 1}, N \zeta_i) > 8 \delta |\zeta_i - \zeta_{i - 1}| N \big] \\
			& \qquad + \mathbb{P} \big[ \dist_{\mathcal{G}} (N \zeta_{k - 1}, N \zeta_k) > 4 \delta N (\diam \mathfrak{R} + 1) \big] < \delta < \varepsilon.
			\end{aligned}
		\end{flalign}
	
		\noindent Here, for the first inequality, we used the fact that
		\begin{flalign*}
			\dist_{\mathcal{G}} (v, \partial R_N) \le \dist_{\mathcal{G}} (v, N \zeta_k) + \dist_{\mathcal{G}} (N \zeta_k, \partial R_N) \le \dist_{\mathcal{G}} (v, N \zeta_k) + 2 \delta N,
		\end{flalign*}
	
		\noindent which holds for large $N$ since $\zeta_k \in \mathfrak{R} \setminus \mathfrak{R}_{\delta}$. For the second estimate, we used the union bound and the fact that
		\begin{flalign*}
			 \displaystyle\sum_{i = 1}^{k - 1} 8 \delta |\zeta_i - \zeta_{i - 1}| + 4 \delta (\diam \mathfrak{R} + 1) \le 16 |\zeta_{k - 1} - \zeta_0| + 4 \delta (\diam \mathfrak{R} + 1) \le 20 \delta (\diam \mathfrak{R} + 1) \le \varepsilon - 2 \delta,
		\end{flalign*}
	
		\noindent where the first bound holds since $\zeta_i - \zeta_{i - 1}$ is for each $i \in [1, k - 1]$ a positive multiple of some vector among $\big\{ (1, 0), (0, 1), (1, 1) \big\}$; the second holds since $|\zeta_{k - 1} - \zeta_0| \le \diam \mathfrak{R}$; and the third holds since $25 \delta (\diam \mathfrak{R} + 1) < \varepsilon$. Since \eqref{vprobabilityr} verifies the proposition, it remains to establish \eqref{pguv}, which we do by analyzing each case for $z_j$ individually.
		
{\bf Case 1:} $z_j$ is of type 1, meaning that $\mathfrak{U} (z_j)$ is the $\delta$-neighborhood of the line $\ell = \ell (z_j, \zeta_{i + 1})$ connecting $z_j$ to the point $z' \in \partial \mathfrak{R}$, and that $z' - \frac{\delta}{2} w = \zeta_{i + 1} = \tfrac{1}{N} \widetilde{v}$. Here the analysis is based on Lemma \ref{distanceg2}.

Denoting $z' - z_j = tw$ for some $t \in \mathbb{R}_{> 0}$ and $w \in \big\{ (1, 0), (0, 1), (1, 1) \big\}$, \Cref{nablaf} implies that $\mathfrak{H} (\zeta_{i + 1}) - \mathfrak{H} (z_j) = 0$ if $w \in \big\{ (1, 0), (0, 1) \big\}$ and $\mathfrak{H} (\zeta_{i + 1}) - \mathfrak{H} (z_j) = t - \frac{\delta}{2}$ if $w = (1, 1)$. Let $\widetilde{z} = z_j + sw \in \ell$, for some $s \in [0, t]$, denote the point on $\ell$ closest to $\tfrac{1}{N} \widetilde{u}$. Then, \Cref{Lemma_var_principle} implies for sufficiently large $N$ that
		\begin{flalign*}
			& \mathbb{P} \Big[ \big| H (\widetilde{v}) - N \mathfrak{H} (\zeta_{i + 1}) \big| \le \delta N \Big] > 1 - \displaystyle\frac{\delta}{4k}; \qquad \mathbb{P} \Big[ \big| H (N \widetilde{z}) - N \mathfrak{H} (\widetilde{z}) \big| \le \delta N \Big] > 1 -  \displaystyle\frac{\delta}{4k}; \\
			& \mathbb{P} \Big[ \big| H' (\widetilde{v}) - N \mathfrak{H} (\zeta_{i + 1}) \big| \le \delta N \Big] > 1 - \displaystyle\frac{\delta}{4k}; \qquad \mathbb{P} \Big[ \big| H' (N \widetilde{z}) - N \mathfrak{H} (\widetilde{z}) \big| \le \delta N \Big] > 1 - \displaystyle\frac{\delta}{4k}.
		\end{flalign*}
	
		\noindent It follows from a union bound that 	
		\begin{flalign*}
			& \mathbb{P} \bigg[ \Big| H(\widetilde{v}) - H(N \widetilde{z}) - N \big( \mathfrak{H} (\zeta_{i + 1}) - \mathfrak{H} (\widetilde{z}) \big) \Big| \le 2 \delta N \bigg] > 1 - \displaystyle\frac{\delta}{2k}; \\
			& \mathbb{P} \bigg[ \Big| H' (\widetilde{v}) - H' (N \widetilde{z}) - N \big( \mathfrak{H} (\zeta_{i + 1}) - \mathfrak{H} (\widetilde{z}) \big) \Big| \le  2 \delta N \bigg] > 1 - \displaystyle\frac{\delta}{2k}.
		\end{flalign*}
			
		\noindent So, since $z'$ is given through \Cref{nablaf} and $\zeta_{i + 1} = z' - \frac{\delta}{2}$, we deduce from a union bound that
		\begin{flalign*}
			\mathbb{P} \big[ \mathcal{F}_1 (N \widetilde{z}, \widetilde{v}; 2 \delta) \cup \mathcal{F}_2 (N \widetilde{z}, \widetilde{v}; 2 \delta) \big] > 1 - \delta k^{-1},
		\end{flalign*}
	
		\noindent where we recall the events $\mathcal{F}_1$ and $\mathcal{F}_2$ from \eqref{f}. Hence, \Cref{distanceg2} gives
		\begin{flalign}
			\label{pgzv}
			\mathbb{P} \big[ \dist_{\mathcal{G}} (N \widetilde{z}, \widetilde{v}) \le 4 \delta N |\widetilde{z} - \zeta_{i + 1}| \big] \ge \mathbb{P} \big[ \dist_{\mathcal{G}} (N \widetilde{z}, \widetilde{v}) \le 4 \delta s N \big] > 1 - \delta k^{-1}.
		\end{flalign}
		
		\noindent Since $\tfrac{1}{N} \widetilde{u} = \zeta_i \in \mathfrak{U} (z_j)$, we have $|\widetilde{z} - \tfrac{1}{N} \widetilde{u}| < \delta$. So,
		\begin{flalign*}
			\dist_{\mathcal{G}} (\widetilde{u}, \widetilde{v}) \le \dist_{\mathcal{G}} (\widetilde{u}, N \widetilde{z}) + \dist_{\mathcal{G}} (N \widetilde{z}, \widetilde{v}) \le |\widetilde{u} - N \widetilde{z}| + \dist_{\mathcal{G}} (N \widetilde{z}, \widetilde{v}) \le \dist_{\mathcal{G}} (N \widetilde{z}, \widetilde{v}) + N \delta,
		\end{flalign*}
	
		\noindent which together with \eqref{pgzv} (and the bound $|\widetilde{z} - \zeta_{i + 1}| \le \diam \mathfrak{R}$) implies the first statement of \eqref{pguv}.
			
\medskip

{\bf Case 2:} $z_j$ is of type 2, in which case $\nabla \mathfrak{H} (z) \in \mathcal{T}_{\delta}$ for each $z \in \overline{\mathfrak{U} (z_j)}$. Here Lemma \ref{distancegestimate2} applies directly. Indeed, because $\zeta_i, \zeta_{i + 1} \in \overline{\mathfrak{U} (z_j)}$ and $|\widetilde{u} - \widetilde{v}| = N |\zeta_i - \zeta_{i + 1}| \ge c_1 N$ for some constant $c_1 = c_1 (\mathfrak{R}, \mathfrak{h}) > 0$, we deduce the second statement of \eqref{pguv} from \Cref{distancegestimate2} applied with the $\varepsilon$ there equal to $\delta k^{-1}$.
		
\medskip

{\bf Case 3:} $z_j$ is of type 3. In this situation we again rely on Lemma \ref{distanceg2}. Without loss of generality we assume that $z_j$ is of type 3(a), as the cases when it is of type 3(b) or type 3(c) are very similar. Then, $\nabla \mathfrak{H} (z)$ is of distance at most $3 \delta$ from the segment $\big[ (0, 0), (1, 0) \big] \subset \partial \mathcal{T}$ for any $z \in \overline{\mathfrak{U} (z_j)}$. Hence, since $\zeta_{i + 1} - \zeta_i$ is a multiple of $(0, 1)$, we have $\big| (\zeta_{i + 1} - \zeta_i) \cdot \nabla \mathfrak{H} (z) \big| \le 3 \delta$ for each $z \in \mathfrak{U} (z_j)$. Integrating $z$ from $\zeta_i$ to $\zeta_{i + 1}$ thus gives $\big| \mathfrak{H} (\zeta_i) - \mathfrak{H} (\zeta_{i + 1}) \big| < 3 \delta |\zeta_i - \zeta_{i + 1}|$. Hence, \Cref{Lemma_var_principle} implies for sufficiently large $N$ that
		\begin{flalign*}
			\mathbb{P} \big[ H(\widetilde{v}) - H (\widetilde{u}) < 4 \delta |\widetilde{u} - \widetilde{v}| \big] \ge 1 - \displaystyle\frac{\delta}{2k}; \qquad \mathbb{P} \big[ H' (\widetilde{v}) - H' (\widetilde{u}) \le 4 \delta |\widetilde{u} - \widetilde{v}| \big] \ge 1 - \displaystyle\frac{\delta}{2k}.
		\end{flalign*}
	
		\noindent Recalling the event $\mathcal{F}_1$ from \eqref{f}, it follows from a union bound that $\mathbb{P} \big[ \mathcal{F}_1  (\widetilde{u}, \widetilde{v}; 4 \delta) \big] \ge 1 - \delta k^{-1}$, and so the second statement of \eqref{pguv} follows from \Cref{distanceg2}. This verifies \eqref{pguv} in all cases; as mentioned above, this implies the proposition.
	\end{proof}

	We have now collected all the ingredients for the proof of \Cref{varianceh2}.
	
	\begin{proof}[Proof of \Cref{varianceh2}]
		
		Fix $\delta > 0$ such that $\delta (1 + \diam \mathfrak{R}) = \frac{\varepsilon}{2}$, and fix a vertex $u \in \partial R_N$ closest to $v$. Then, for sufficiently large $N$ we have
		\begin{flalign*}
			\Var \big[ F(v) \big] = \mathbb{E} \big[ \dist_{\mathcal{G}} (u, v) \big] & \le \delta N + \mathbb{P} \big[ \dist_{\mathcal{G}} (u, v) > \delta N \big] \diam R_N \\
			& \le \delta N + \delta \diam R_N \le \delta N (1 + \diam \mathfrak{R}) + \displaystyle\frac{\delta N}{4} < \varepsilon N.
		\end{flalign*}
		
		\noindent Here, the first statement holds by \Cref{uvf} and the deterministic fact that $F(u) = 0$ (as $u \in \partial R_N$); the second by the deterministic bound $\dist_{\mathcal{G}} (u, v) \le |u - v| \le \diam R_N$; the third by \Cref{distanceguv}; the fourth by the fact that $\tfrac{1}{N} \diam R_N \le \diam \mathfrak{R} + \frac{\delta}{4}$ holds for sufficiently large $N$ (since $\lim_{N \rightarrow \infty} \tfrac{1}{N} R_N = \mathfrak{R}$); and the fifth by the definition of $\delta$. This establishes the proposition.	
	\end{proof}

\section{Convergence to the GUE--corners process: proof}

\subsection{Tilings of fixed trapezoids}

While the first ingredient of our proof of Theorem \ref{Theorem_GUE_general} is Proposition \ref{Proposition_varianceh}, the second equally important component is the asymptotic analysis of lozenge tilings of trapezoids from \cite{GP,Novak}. Let us recall the result from these texts which we need. By a \emph{fixed trapezoid} we mean a domain of Figure \ref{Figure_trap}: it is parameterized by the width $L$, the length $A$ of the vertical boundary, and by $L$--tuple of integers $0\le \lambda_1< \lambda_2< \dots< \lambda_L < A+L$, which encodes the positions of the horizontal lozenges along the right boundary of the domain. We refer to these $L$ lozenges as \emph{dents}. The lozenge tilings of a fixed trapezoid are in bijection with triangular arrays $(y_i^k)_{1\le i \le k \le L}$ of $L(L+1)/2$ integers, which satisfy:
\begin{equation}
\label{eq_interlacing}
 y_i^{k+1} \le y_{i}^{k}< y_{i+1}^{k+1}, \quad 1\le i \le k <L,\qquad \text{ and }\quad  (y_1^L,\dots,y_L^L)=(\lambda_1,\dots,\lambda_L).
\end{equation}
The integers of the array encode positions of horizontal lozenges, as shown in Figure \ref{Figure_trap}. The role of the parameter $A$ for the array is very limited: it only appears in the inequality $\lambda_L<A+L$. However, this parameter becomes more important when we want to reconstruct a tiling: the same array can lead to slightly different tilings depending on the value of $A$.

\begin{figure}[t]
\center
  \includegraphics[width=0.8\textwidth]{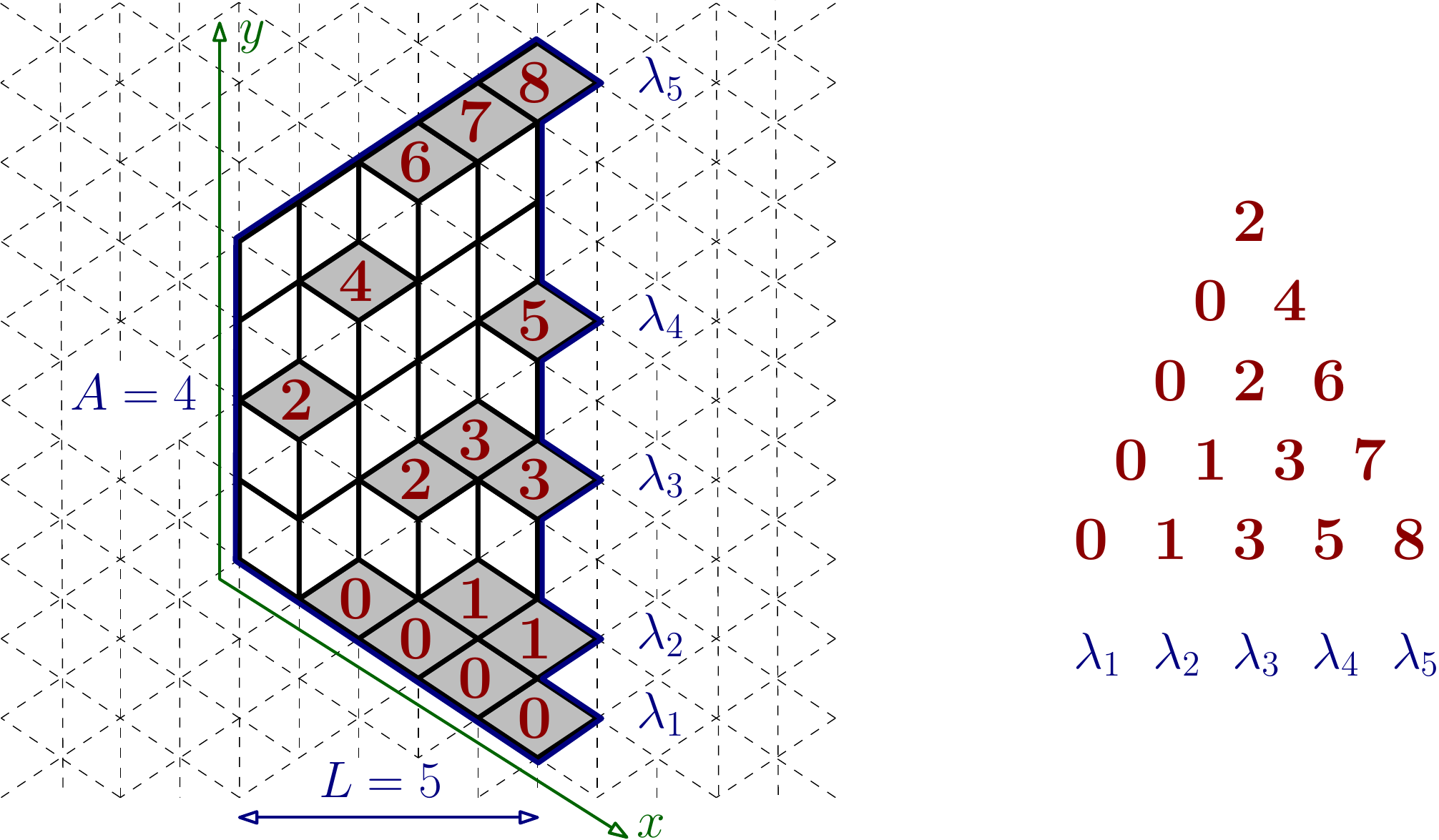}
\caption{A tiling of a fixed trapezoid and corresponding triangular array of coordinates of horizontal lozenges}
\label{Figure_trap}
\end{figure}

We are interested in the behavior of uniformly random lozenge tilings of a fixed trapezoid with dents $\lambda_1,\dots,\lambda_L$ as $L$ and $A$ tend to infinity. In this situation $(y_i^k)_{1\le i \le k \le L}$ become random variables and we would like to understand their asymptotics. For finite $i$ and $k$ the limiting behavior is summarized in the following lemma, which is a slight reformulation of results from \cite{GP} and \cite{Novak}.

\begin{lemma}[{\cite[Theorem 1.7]{GP},\cite[Theorem 1]{Novak}}] \label{Lemma_GUE_trap}
  Suppose that $\lambda=(\lambda_1,\dots,\lambda_L)$ and $A$ depend on $L$. Then, in the sense of convergence in finite-dimensional distributions,  we have
 \begin{equation}
 \label{eq_conv_GUE_trap}
  \lim_{L\to\infty} \frac{y_i^k-m(\lambda)}{\sigma(\lambda)\sqrt{L}}= \xi_i^k, \qquad \text{jointly over }1\le i \le k,
 \end{equation}
 where $\xi_i^k$ is the GUE--corners process of Definition \ref{Def_GUE_corners},
 \begin{equation}
 \label{eq_params_trap}
  m(\lambda)=\left[\sum_{i=1}^L \frac{\lambda_i}{L}\right]- \frac{L}{2},\qquad \sigma(\lambda)^2= \frac{1}{L}\sum_{i=1}^L \left(\frac{\lambda_i}{L}\right)^2 - \left( \frac{1}{L} \sum_{i=1}^L \frac{\lambda_i}{L}\right)^2 - \frac{1}{12},
 \end{equation}
 and the convergence is uniform over $A$ such that $A/L$ stays bounded as $L\to\infty$ and $\lambda$ such that $\sigma(\lambda)$ stays bounded away from $0$ as $L\to\infty$. If instead $\sigma(\lambda)$ tends to $0$ as $L\to\infty$, then \eqref{eq_conv_GUE_trap} should be interpreted as convergence of $\frac{y_i^k-m(\lambda)}{\sqrt{L}}$ to the zero vector.
\end{lemma}
\begin{remark}
 The minimal value that $\sigma(\lambda)^2$ can take is achieved when $(\lambda_1,\dots,\lambda_L)=(B,B+1,\dots,B+L-1)$. In this situation the lozenge tiling is completely frozen (i.e.\ there exists only one array satisfying \eqref{eq_interlacing}) and
 $$
  \sigma(\lambda)^2= \frac{L(L-1)(2L-1)}{6L^3}- \left( \frac{L(L-1)}{2L^2}\right)^2-\frac{1}{12}= \frac{L-1}{L}\left( \frac{2L-1}{6L}- \frac{L-1}{4L}\right)-\frac{1}{12}\stackrel{L\to\infty}{\longrightarrow}0.
 $$
 On the other hand, if we assume that the empirical measures $\frac{1}{L}\sum_{i=1}^L \delta_{\lambda_i/L}$ converge as $L\to\infty$ to a probability measure\footnote{$\lambda_{i+1}-\lambda_i\ge 1$ implies that the Lebesgue density of the limiting measure is at most one.} other than a uniform measure on an interval of length one, then $\sigma^{2}$ converges to a positive constant.
\end{remark}

\subsection{Proof of Theorem \ref{Theorem_GUE_general}}

In Lemma \ref{Lemma_GUE_trap} we dealt with fixed trapezoids, which means that the $L$--tuples of dents $(\lambda_1,\dots,\lambda_L)$ were deterministic. Let us now introduce a notion of \emph{fluctuating trapezoid}: this is a trapezoid in which the dents are allowed to be random. When we discuss tilings of a fluctuating trapezoid, we assume that the law of the dents can be arbitrary, but conditionally on the positions of the dents, the probability distribution on lozenge tilings inside the trapezoid is uniform.

For each $N=1,2,\dots$, let us look at the embedded trapezoid of Theorem \ref{Theorem_GUE_general}. With probability tending to $1$ as $N\to\infty$, each tiling of $R_N$ gives a rise to a lozenge tilings of the (embedded) fluctuating trapezoid. Let us denote the width of this trapezoid through $L=L(N)$ and its dents through $\lambda_i=\lambda_i(N)$, $1\le i \le L$. We would like to set $m_N=\E m(\lambda)$ and $\sigma^2=\lim_{N\to\infty} \sigma(\lambda)^2$, as given by \eqref{eq_params_trap}, and apply Lemma \ref{Lemma_GUE_trap} to get \eqref{eq_gen_to_GUE} (note that $A/L$ in Lemma \ref{Lemma_GUE_trap} is uniformly bounded, since all the domains involved in Theorem \ref{Theorem_GUE_general} and the ratio $N/L$ are bounded). The caveat is that \eqref{eq_conv_GUE_trap} involves \emph{random} $m(\lambda)$ and $\sigma(\lambda)$, rather than deterministic $m_N$ and $\sigma$ of \eqref{eq_gen_to_GUE}. Hence, we need to prove the following two claims:

\smallskip

\noindent{\bf Claim I:} \,\,\, $\displaystyle\lim_{N\to\infty}\frac{m(\lambda)-\E m(\lambda)}{\sqrt{N}}=0$, in probability.

\noindent {\bf Claim II:} \, There exists a deterministic limit in probability $\displaystyle \lim_{N\to\infty}  \sigma(\lambda)^2$.

\smallskip

For Claim I we need to rewrite $m(\lambda)$ in terms of the height function. Let $H_L(y)$ denote the height function of the tiling along the vertical line passing through the dents of the embedded trapezoid. In more detail, following the notation of Section \ref{Section_height_function}, we choose the origin of the coordinate system to be in the bottom--right corner of the embedded trapezoid and set $H_L(y):=H(0,y)$.
 As the local rules of Figure \ref{Figure_height_function} show, this function starts as a constant $C$ at the bottom of the trapezoid and as we move up it grows by $1$ whenever we cross a horizontal lozenge and stays constant otherwise.
Hence, we can write
\begin{equation}
\label{eq_sum_as_H}
\sum_{i=1}^L \lambda_i = \sum_{y=0}^{A+L-1} y \bigl(H_L(y+1)-H_L(y)\bigr)=(A+L-1) H_L(A+L)-\sum_{y=1}^{A+L-1} H_L(y).
\end{equation}
Because $H_L(A+L)$ is deterministic, we conclude that
\begin{equation}
\label{eq_m_transform}
 m(\lambda)-\E m(\lambda)=\frac{1}{L}\left(\sum_{i=1}^L \lambda_i-\E\sum_{i=1}^L \lambda_i\right)=-\frac{1}{L}\sum_{y=1}^{A+L-1}\bigl(H_L(y)-\E H_L(y)\bigr)
\end{equation}
Using a bound on the variance of a sum of arbitrary random variables
$$
 \Var(\alpha_1+\alpha_2+\dots+\alpha_K)\le K \bigl( \Var \alpha_1+\Var \alpha_2+\dots +\Var \alpha_K),
$$
 the result of Proposition \ref{Proposition_varianceh}, and \eqref{eq_m_transform}, we get as $N\to\infty$
$$
 \E \bigl(m(\lambda)-\E m(\lambda)\bigr)^2\le \frac{A+L-1}{L^2} \sum_{y=1}^{A+L-1} \Var H_L(y)\le \frac{(A+L-1)^2}{L^2} o(N).
$$
Because both $L$ and $A$ linearly depend on $N$, the last bound and the Markov's inequality imply {Claim~I}.

For Claim II we notice that \eqref{eq_sum_as_H} together with Lemma \ref{Lemma_var_principle} implies that $\sum_{i=1}^{L} \lambda_i$ can be approximated using the limit shape given by the variational principle. Recalling that we have (without loss of generality) assumed that the origin of the coordinate system is at the bottom-right corner of the embedded trapezoid, we have\footnote{In \eqref{eq_sum_limit} we silently assume that entire segment $(0,0)-(0, \tfrac{A+L}{N})$ lies in $\mathfrak R$. This segment has to be inside $R_N$ by definition, but $\mathfrak R$ might be slightly different from $R_N$. However, this difference only introduces another $o(1)$ error, which we can ignore.}
\begin{equation}
\label{eq_sum_limit}
 \frac{1}{L}\sum_{i=1}^{L}\frac{\lambda_i}{L}=\frac{N(A+L-1)}{L^2}  \mathfrak H\left(0, \tfrac{A+L}{N}\right)-\frac{N}{L^2}\sum_{y=1}^{A+L-1} \mathfrak H\left(0, \tfrac{i}{N}\right)+o(1),
\end{equation}
where $\mathfrak H$ is the limit shape and $o(1)$ is random asymptotically vanishing term as $N\to\infty$. The first term in the right-hand side of \eqref{eq_sum_limit} has a straightforward $N\to\infty$ limit, while the second one is a Riemann sum, approximating the integral of $\mathfrak H$ along the right boundary of the embedded trapezoid as $N\to\infty$. We conclude that \eqref{eq_sum_limit} converges to a deterministic constant as $N\to\infty$. A very similar computation shows that
$$
 \frac{1}{L}\sum_{i=1}^L \left(\frac{\lambda_i}{L}\right)^2
$$
also has a deterministic limit as $N\to\infty$, which can be directly expressed through the limit shape $\mathfrak H$. We conclude that all terms in the definition of $\sigma(\lambda)^2$ converge as $N\to\infty$, which is the desired statement of Claim II.

The proof of Theorem \ref{Theorem_GUE_general} is finished.

\appendix

\section{Uniformity in Convergence of Local Statistics}

\label{ProofConverge}

In this section we explain how \Cref{convergemu} is a consequence of results from \cite{ULS}. In what follows, we let $\mathcal{B}_R (z) = \big\{ z' \in \mathbb{R}^2 : |z' - z| = R \big\}$ denote the disk of radius $R$ centered at $z \in \mathbb{R}^2$.

\begin{proof}[Proof of \Cref{convergemu}]
	
	As mentioned in \Cref{convergev}, \Cref{convergemu} essentially coincides with Theorem 1.5 of \cite{ULS}, except that the uniformity of the constant $C>1$ in the vertex $v \in R_N$ was not explicitly claimed there. On the other hand, it was stated in \cite{ULS} that such uniformity holds if $R_N$ is a disk containing $v$. More specifically, suppose that there exists some vertex $u \in \mathbb{T}$ such that the following holds.
	\begin{enumerate}
		\item $R_N$ approximates a disk $\mathcal{B}_N (u)$ centered at $u$. That is, we have $R_N = \mathcal{B}_N (u) \cap \mathbb{T}$.
		\item  We have $\mathcal{B}_{\varepsilon N} (v) \subset R_N$ (that is, $R_N$ contains a disk centered at $v$).
		\item For each $z \in \mathcal{B}_{\varepsilon} (\tfrac{1}{N} v)$, we have $\nabla \mathcal{H} (z) \in \mathcal{T}_{\varepsilon}$ (that is, $\nabla \mathcal{H}$ is uniformly liquid on that disk).
	\end{enumerate}

	\noindent  Then, Theorem 3.15 (see also Assumption 3.5) of \cite{ULS} states that \Cref{convergemu} holds, with $C$ uniform in $v$. Let us mention that Theorem 3.15 of \cite{ULS} does not require $R_N$ to be tileable. If it is not, then $H_N$ is a uniformly random height function on $R_N$ with (some) boundary height function $h_N$ defined on $\partial R_N$, and $\mathscr{M}$ is the unique \emph{free tiling} associated with $H_N$, in which tiles are permitted to extend beyond the boundary of $R_N$.
	
	Although in \Cref{convergemu} the domain $R_N$ does not have to approximate a disk, we may apply the above result on a subdisk of it containing $v$ in the following way. Since $\nabla \mathcal{H} (\mathfrak{v}) \in \mathcal{T}_{\varepsilon}$,  \cite[Proposition 4.1]{MCFRS} implies that $\nabla \mathcal{H}$ is continuous in a neighborhood of $\mathfrak{v} = \tfrac{1}{N} v \in \mathfrak{R}$. Letting $\kappa = \kappa (\varepsilon) > 0$ denote a sufficiently small real number to be fixed later, there then exists $\rho = \rho (\varepsilon, \kappa, \mathfrak{R}) \in (0, \varepsilon)$ such that $B_{\rho} (\mathfrak{v}) \subset \mathfrak{R} \cap \tfrac{1}{N} R_N$ and
	\begin{flalign}
		\label{hzkappa}
		\big| \nabla \mathcal{H} (z) - \nabla \mathcal{H} (\mathfrak{v}) \big| < \kappa, \qquad \text{for each $z \in \mathcal{B}_{\rho} (\mathfrak{v})$}.
	\end{flalign}

	\noindent Let $\widetilde{R}_N = \mathcal{B}_{\rho N} (v) \cap \mathbb{T}$ and $\widetilde{\mathfrak{R}}_N = \tfrac{1}{N} \widetilde{R}_N$, and condition on the restriction $\mathscr{M}_N |_{R_N \setminus \tilde{R}_N}$ of the random tiling $\mathscr{M}_N$ to $R_N \setminus \widetilde{R}_N$. This induces (random) boundary data $\widetilde{\mathfrak{h}}_N: \partial \widetilde{\mathfrak{R}}_N \rightarrow \mathbb{R}$ defined by setting $\widetilde{\mathfrak{h}}_N (z) = \tfrac{1}{N} H_N (N z)$ for each $z \in \partial \widetilde{\mathfrak{R}}_N$. Let $\widetilde{\mathcal{H}} \in \Adm (\widetilde{\mathfrak{R}}_N; \widetilde{\mathfrak{h}}_N)$ denote the maximizer of $\mathcal{E}$ on $\widetilde{\mathfrak{R}}_N$ with boundary data $\widetilde{\mathfrak{h}}_N$. In order to apply Theorem 3.15 of \cite{ULS} on the domain $\widetilde{\mathfrak{R}}_N$ with boundary data $\widetilde{\mathfrak{h}}_N$, we must verify that $\nabla \widetilde{\mathcal{H}}$ is likely uniformly liquid around $\mathfrak{v}$ (for example, $\nabla \widetilde{\mathcal{H}} (z) \in \mathcal{T}_{\varepsilon / 2}$ for $z \in \mathcal{B}_{\rho  / 4} (\mathfrak{v})$ with high probability).
	
	To that end, observe for any fixed $\delta = \delta (\varepsilon, \omega, D) > 0$ that the variational principle, \Cref{Lemma_var_principle}, yields a constant $C_0 = C_0 (\delta, \mathfrak{R}, \mathfrak{h}) > 1$ such that for $N > C_0$ we have
	\begin{flalign}
		\label{e}
		\mathbb{P} [\mathcal{E}_N] > 1 - \delta, \qquad \text{where} \qquad  \mathcal{E}_N = \mathcal{E}_N (\delta) = \bigg\{ \displaystyle\sup_{z \in \partial \widetilde{\mathfrak{R}}_N} \big| \widetilde{\mathfrak{h}}_N (z) - \mathcal{H} (z) \big| < \delta \bigg\}.
	\end{flalign}

	\noindent Next, for sufficiently small $\kappa = \kappa (\varepsilon) > 0$, \eqref{hzkappa} and Proposition 2.13 of \cite{ULS} together yield a constant $C_1 = C_1 (\varepsilon) > 1$ such that
	\begin{flalign}
		\label{deltae}
		\textbf{1}_{\mathcal{E}} \displaystyle\sup_{z \in \mathcal{B}_{\rho / 4} (\mathfrak{v})} \big| \nabla \widetilde{\mathcal{H}} (z) - \nabla \mathcal{H} (z) \big| < C_1 \textbf{1}_{\mathcal{E}} \displaystyle\sup_{z \in \partial \widetilde{\mathfrak{R}}_N} \big| \widetilde{\mathfrak{h}}_N (z) - \mathcal{H} (z) \big| < C_1 \delta.
	\end{flalign}
	
	\noindent Hence, if $\delta = \delta (\varepsilon, \omega, D) > 0$ and $\kappa = \kappa (\varepsilon)$ are chosen sufficiently small so that $C_1 \delta < \frac{\varepsilon}{4}$ and $\kappa < \frac{\varepsilon}{4}$, then it follows from \eqref{hzkappa} and \eqref{deltae} that $\nabla \widetilde{\mathcal{H}} (z) \in \mathcal{T}_{\varepsilon / 2}$ for each $z \in \mathcal{B}_{\rho / 4} (\mathfrak{v})$ (since $\nabla \mathcal{H} (\mathfrak{v}) \in \mathcal{T}_{\varepsilon}$).
	
	In particular, Theorem 3.15 of \cite{ULS} applies. Denoting $(\widetilde{s}, \widetilde{t}) = \nabla \widetilde{\mathcal{H}} (\mathfrak{v})$, it yields a constant $C_2 = C_2 (\varepsilon, \omega, D, \mathfrak{R}, \mathfrak{h}) > 1$ such that for $N > C_2$ we have
	\begin{flalign}
		\label{e1}
		\textbf{1}_{\mathcal{E}} d_{\TV} \big(\mathscr{M} |_{\mathcal{B}_D (v)}, \mu_{\tilde{s}, \tilde{t}} |_{\mathcal{B}_D (v)} \big) < \frac{\omega}{2}.
	\end{flalign}	
	
	\noindent By \eqref{deltae}, we have for sufficiently small $\delta = \delta (\varepsilon, \omega, D) > 0$ that
	\begin{flalign}
		\label{e2}
		\textbf{1}_{\mathcal{E}} d_{\TV} (\mu_{\tilde{s}, \tilde{t}} |_{\mathcal{B}_D (v)}, \mu_{s, t} |_{\mathcal{B}_D (v)}) < \frac{\omega}{4}.
	\end{flalign}
	
	\noindent Thus, by \eqref{e}, \eqref{e1}, \eqref{e2}, and further imposing that $\delta < \frac{\omega}{4}$, we have for $N > \max \{ C_0, C_2 \}$ that
	\begin{flalign*}
		d_{\TV} (\mathscr{M} |_{\mathcal{B}_D (v)}, \mu_{s, t} |_{\mathcal{B}_D (v)}) & \le \textbf{1}_{\mathcal{E}} d_{\TV} (\mathscr{M} |_{\mathcal{B}_D (v)}, \mu_{s, t} |_{\mathcal{B}_D (v)}) + \mathbb{P} [\mathcal{E}^c] \\
		& < \textbf{1}_{\mathcal{E}} d_{\TV} (\mathscr{M} |_{\mathcal{B}_D (v)}, \mu_{\tilde{s}, \tilde{t}} |_{\mathcal{B}_D (v)}) + d_{\TV} (\mu_{\tilde{s}, \tilde{t}} |_{\mathcal{B}_D (v)}, \mu_{s, t} |_{\mathcal{B}_D (v)}) + \frac{\omega}{4} < \omega,
	\end{flalign*}

	\noindent which implies the lemma.
\end{proof}


\end{document}